\documentclass[11pt,twoside]{article}
\usepackage{}
\usepackage{undertilde}
\usepackage{amsmath}
\usepackage{amssymb}
\usepackage{fancyhdr}
\usepackage{latexsym}
\usepackage{bbding}
\usepackage{mathrsfs}
\usepackage{wasysym}
\usepackage{cite}
\usepackage{color}
\usepackage{multicol,graphics}

\newtheorem{theorem}{Theorem}[section]
\newtheorem{lemma}[theorem]{Lemma}
\newtheorem{corollary}[theorem]{Corollary}

\newtheorem{definition}[theorem]{Definition}
\newtheorem{remark}[theorem]{Remark}
\newtheorem{proposition}[theorem]{Proposition}
\numberwithin{equation}{section}
\newenvironment{proof}[1][Proof]{\noindent\textbf{#1.} }{\hfill $\Box$}
\allowdisplaybreaks

 \makeatletter
\begin{document}

\title{\textbf{A Nonlocal Dispersal SIS Epidemic Model in Heterogeneous Environment}}
\author{Fei-Ying Yang,~ Wan-Tong Li\thanks{%
Corresponding author (wtli@lzu.edu.cn).}~
and Liang Zhang \\
School of Mathematics and Statistics,\\
Key Laboratory of Applied Mathematics and Complex Systems,\\
Lanzhou University, Lanzhou, Gansu, 730000, P.R. China}
%
\maketitle

\begin{abstract}
This article is concerned with a nonlocal dispersal susceptible-infected-susceptible (SIS) epidemic model with Neumann boundary condition, where the rates of disease transmission and recovery are assumed to be spatially heterogeneous and the total population number is constant. We first introduce the basic reproduction number $R_0$ and then discuss the existence, uniqueness and stability of steady states of the nonlocal dispersal SIS epidemic model in terms of $R_0$. In particular, we also consider the impacts of the large diffusion rates of the susceptible and infected population on the persistence and extinction of the disease, and these results imply that the nonlocal movement of the susceptible or infected individuals will enhance the persistence of the disease. Additionally, our analytical results also suggest that the spatial heterogeneity tends to boost the spread of the disease. We have to emphasize that the main difficulty is that a lack of regularizing effect occurs.



\end{abstract}


\section{Introduction}\label{Int}

\noindent

This paper is concerned with the following nonlocal dispersal SIS epidemic model
\begin{equation}\label{101}
\begin{cases}
\frac{\partial S}{\partial t}=d_S\int_{\Omega}J(x-y)(S(y,t)-S(x,t))dy
-\frac{\beta(x)SI}{S+I}+\gamma(x)I,& x\in\Omega,~t>0,\\
\frac{\partial I}{\partial t}=d_I\int_{\Omega}J(x-y)(I(y,t)-I(x,t))dy
+\frac{\beta(x)SI}{S+I}-\gamma(x)I, & x\in\Omega,~t>0,\\
S(x,0)=S_0(x),~I(x,0)=I_0(x), & x\in\Omega,
\end{cases}
\end{equation}
where, $\Omega$ is a bounded domain; $S(x,t)$ and $I(x,t)$ denote the density of susceptible and infected individuals at location $x$ and time $t$, respectively; $d_S$ and $d_I$ are positive diffusion coefficients for the susceptible and infected individuals; $\beta(x)$ and $\gamma(x)$ are positive continuous functions on $\bar{\Omega}$ that represent the rates of disease transmission and recovery at $x$. It is known from \cite{ABL-2008} that the term $\frac{SI}{S+I}$ is a Lipschitz continuous function of $S$ and $I$ in the open first quadrant, we can extend its definition to the entire first quadrant by defining it to be zero when either $S=0$ or $I=0$. Throughout the whole paper, we assume that the initial infected individuals are positive without other description, that is
\begin{equation*}
\int_{\Omega}I(x,0)dx>0,~\text{with}~S_0(x)\ge0~\text{and}
~I_0(x)\ge0~\text{for}~x\in\Omega
\end{equation*}
and the dispersal kernel function $J$ satisfies
\begin{description}
\item(J) $J(x)\in C(\bar{\Omega}),J(0)>0,~J(x)=J(-x)\ge0,
~\int_{\mathbb{R}^N}J(x)dx=1$
and $\int_{\Omega}J(x-y)dy\not\equiv1$ for any $x\in\Omega$.
\end{description}

By the standard theory of semigroups of linear bounded operator \cite{Pazy}, we know from \cite{KLSH2010} that \eqref{101} admits a unique positive solution $(S_*(x,t),I_*(x,t))$, which is continuous with respect to $x$ and $t$. Let
\begin{equation*}
N:=\int_{\Omega}(S_0(x)+I_0(x))dx.
\end{equation*}
If we add the first equation and the second equation of \eqref{101} and integrate it on $\Omega$, then there is
\begin{equation*}
\frac{\partial}{\partial t}\int_{\Omega}(S(x,t)+I(x,t))dx=0~~\text{for all}~t>0.
\end{equation*}
This implies that the total population size is constant, that is
\begin{equation*}
\int_{\Omega}(S(x,t)+I(x,t))dx=N~~\text{for all}~t\ge0.
\end{equation*}

Note that \eqref{101} is the nonlocal counterpart of the following SIS epidemic reaction-diffusion model
\begin{equation}\label{104}
\begin{cases}
\frac{\partial S}{\partial t}=d_S\Delta S-\frac{\beta(x)SI}{S+I}
+\gamma(x)I,~& x\in\Omega, t>0,\\
\frac{\partial I}{\partial t}=d_I\Delta I+\frac{\beta(x)SI}{S+I}
-\gamma(x)I,~& x\in\Omega, t>0,\\
\partial_{\nu}S=\partial_{\nu}I=0,~& x\in\partial\Omega, t>0,
\end{cases}
\end{equation}
in which $\nu$ is the outward unit normal vector on $\partial\Omega$. As is well known, \eqref{104} is first proposed and considered by Allen et al. \cite{ABL-2008}, where they mainly discussed the impact of spatial heterogeneity of environment and movement of individuals on the persistence and extinction of a disease. After that, their results were extended by Peng and Liu \cite{Peng2009}, in which they proved that the endemic equilibrium was globally asymptotically stable if it exists and this result confirms the conjecture proposed in \cite{ABL-2008}. Meanwhile, Peng \cite{Peng20091} provided further understanding regarding the roles of large or small migration rates of the susceptible and infected population on the spatial persistence and extinction of the epidemic disease, and these results were further extended in \cite{Peng2013}. Moreover, Peng and Zhao in \cite{Peng2012} considered \eqref{104} when the rates of disease transmission and recovery are assumed to be spatially heterogeneous and temporally periodic. For other results about SIS epidemic models with spatial heterogeneity, one can also see \cite{ABL-2007,ALL-2009,Huang2010}.

The nonlocal dispersal as a long range process can be better to describe some natural phenomena in many situations \cite{AMRT2010,Fife2003}. Thus, nonlocal dispersal equations have attracted much attention and have been used to model different dispersal phenomena in population ecology, material science and neurology, see \cite{Be2015,HMMV2003,KLSH2010,RN2012,Wang2002,bates,SLY2014}. For the study of the nonlocal problem, we refer to
\cite{CCR2006,CCEM2007,SLY2011,ZLS2010} about the asymptotic behavior,
\cite{SHZH2010,BFRW1997,SLW2010,Coville-2013,PLL2009,LSW2010,LZZ2015} about the
traveling waves and entire solutions when $\Omega=\mathbb{R}$ and
\cite{CDM2008,BZH2007,SLW2012,Coville-2015,YLS2016,Wu2015} about the stationary solutions. In particular, the spectrum properties of nonlocal dispersal operators and their essential difference comparing with the random operators are contained
in \cite{Coville2010,Coville-2013,MR2009,SLY2014,SHZH2010,XS2015}.
Nowadays, the diffusion process is described by an integral operators such as $J*u-u=\int_{\mathbb{R}^N}J(x-y)u(y)dy-u$ in epidemiology, see \cite{
YYLW2013,YLW2015} and the references therein.

The present paper is devoted to the dynamic behavior of system \eqref{101}. It is well-known that the basic reproduction number denoted by $R_0$ is an important threshold parameter to discuss the dynamic behavior of the epidemic models. For system \eqref{101}, it is natural to ask what the basic reproduction number is and how it decides the dynamic behavior of \eqref{101}. As one of the important quantities in epidemiology, the basic reproduction number $R_0$ of an infectious disease is defined to be the expected number of secondary cases produced, in a completely susceptible population, by a typical infective individual (see, e.g., \cite{HSW2005} and the reference therein). For autonomous epidemic models, Diekmann et al.\cite{Diekmann1990} introduced $R_0$ by using the next generation operators. Driessche and Watmough \cite{Driessche2002} established the theory of $R_0$ for compartment ODE models. Thieme \cite{Thieme2009} further developed a general theory of spectral bound and reproduction number for the infinite-dimensional population structure and time heterogeneity. For a nonlocal and time-delayed reaction-diffusion model of dengue fever, Wang and Zhao \cite{WZ2011} gave the definition of $R_0$ via a next generation operator and proved the threshold dynamics in terms of $R_0$. Recently, Wang and Zhao \cite{WZ2012} presented the theory of $R_0$ for reaction-diffusion epidemic models with compartment structure and in particular, characterized $R_0$ by means of the principal eigenvalue of an elliptic eigenvalue problem. In the current study, motivated by the works in \cite{Thieme2009,WZ2011,WZ2012}, we intend to introduce the basic reproduction number $R_0$ for model \eqref{101}, and give its characterization. We further prove that $R_0-1$ has the same sign as
\begin{eqnarray*}
\mu_p(d_I)=\sup_{\scriptstyle \varphi\in L^2(\Omega) \atop\scriptstyle
\varphi\neq0}\frac{-\frac{d_I}{2}\int_{\Omega}\int_{\Omega}
J(x-y)(\varphi(y)-\varphi(x))^2dydx+\int_{\Omega}(\beta(x)-\gamma(x))
\varphi^2(x)dx}{\int_{\Omega}\varphi^2(x)dx},
\end{eqnarray*}
which is of interest by itself. In general, $\mu_p(d_I)$ may not be the principal eigenvalue of the nonlocal operator
\begin{equation*}
\mathcal{M}[u](x):=d_I\int_{\Omega}J(x-y)(u(y)-u(x))dy
+(\beta(x)-\gamma(x))u(x),
\end{equation*}
which may lead some essential difference between the nonlocal dispersal problem and the reaction-diffusion problem, see \cite{HMMV2003,Coville2010,SHZH2010,SLY2014,SLW2012}.
Thus, we may not characterize the basic reproduction number of system \eqref{101} through the eigenvalue problem corresponding to the nonlocal operators. To overcome this difficulty, we use the basic theory developed in \cite{Thieme2009} to give the definition of the basic reproduction number of system \eqref{101}.

After then, we are most concerned with the global stability of the disease-free equilibrium and the endemic equilibrium of system \eqref{101}. It is obtained that the disease-free equilibrium is unique and globally stable when $R_0<1$, which implies that the disease will be extinct in this case. By the sub-super solution method, we give the existence of the endemic equilibrium, and motivated by the method in \cite{Be2005}, obtain its uniqueness as $R_0>1$. Generally, the stability of the endemic equilibrium is very difficult to be proved. When $d_S=d_I$, however, the global stability of the endemic equilibrium can be shown by constructing some auxiliary problems in this paper. On the basis of Lyapunov stability theorem, the global stability of the endemic equilibrium can be also obtained if the rate of disease transmission is proportional to the rate of the disease recovery. These results imply that the disease will be persistent when $R_0>1$. Finally, we discuss the effect of the diffusion rates $d_S$ and $d_I$ on the disease transmission.
Necessarily, we find that the nonlocal movement of the susceptible or infected individuals tends to enhance the persistence of the disease. However,
there are some difficulties for us to overcome when we prove the above results due to the lack of the regularity of solutions of \eqref{101} or the stationary solutions corresponding to system \eqref{101}.

This paper is organized as follows. In Section 2, we give the characterization of the reproduction number of system \eqref{101}. The existence, uniqueness and global stability of the disease-free equilibrium are obtained in Section 3. Section 4 is devoted to the existence, uniqueness and global stability of the endemic equilibrium. In Section 5, we discuss the effect of the diffusion rates associated with the susceptible and infected individuals on the disease transmission. Finally, we give a brief discussion to complete the paper.

\section{The basic reproduction number}
\noindent

In this section, we will give the definition of the basic reproduction number for system \eqref{101} and provide its analytical properties.
Let $X=C(\bar{\Omega})$ be the Banach space of real continuous functions on $\bar{\Omega}$. Throughout this section, $X$ is considered as an ordered Banach space with a positive cone
$X_+=\{u\in X|~u\ge0\}$. It is well-known that $X_+$ is generating, normal and has nonempty interior. Additionally, an operator $T: X\to X$ is called positive if $TX_+\subseteq X_+$.

\begin{definition}
A closed operator $\mathscr{A}$ in $X$ is called resolvent-positive, if the resolvent set of $\mathscr{A}$, $\rho(\mathscr{A})$, contains a ray $(\omega,\infty)$ and $(\lambda I-\mathscr{A})^{-1}$ is a positive operator for all $\lambda>\omega$.
\end{definition}
\begin{definition}
The spectral bound of $\mathscr{A}$ is defined by
$$S(\mathscr{A})=\sup\{Re\lambda|~\lambda\in\sigma(\mathscr{A})\},$$
where $\sigma(\mathscr{A})$ denotes the spectrum of $\mathscr{A}$. The spectral radius of $\mathscr{A}$ is defined as
$$r(\mathscr{A})=\sup\{|\lambda|; ~\lambda\in\sigma(\mathscr{A})\}.$$
\end{definition}

\begin{theorem}[\cite{Thieme2009}]\label{theorem203}
Let $\mathscr{A}$ be the generator of a $C_0-$semigroup $S$ on the ordered Banach space $X$ with a normal and generating cone $X_+$. Then, $\mathscr{A}$ is a resolvent-positive if and only if $S$ is a positive semigroup, i.e., $S(t)X_+\subset X_+$ for all $t\ge0$. If $\mathscr{A}$ is resolvent-positive, then
\begin{equation*}
(\lambda I-\mathscr{A})^{-1}x=\lim_{b\to\infty}\int_{0}^{b}e^{\lambda t}S(t)xdt,~\lambda>S(\mathscr{A}),~x\in X.
\end{equation*}
\end{theorem}

\begin{theorem}[\cite{Thieme2009}]\label{theorem204}
Let $\mathscr{B}$ be a resolvent-positive operator on $X$, $S(\mathscr{B})<0$ and $\mathscr{A}=\mathscr{C}+\mathscr{B}$ a positive perturbation of $\mathscr{B}$. If $\mathscr{A}$ is resolvent-positive, $S(\mathscr{A})$ has the same sign as $r(-\mathscr{C}\mathscr{B}^{-1})-1$.
\end{theorem}

For applying the basic theory in \cite{Thieme2009} to discuss the basic reproduction number of system \eqref{101}, we first consider the eigenvalue problem
\begin{equation}\label{201}
\mathcal{M}[u](x):=d_I\int_{\Omega}J(x-y)(u(y)-u(x))dy+\beta(x)u(x)-\gamma(x)u(x)
=-\lambda u(x)~~\text{in}~\Omega,
\end{equation}
which will be also used to obtain the main result in this section.
Define
\begin{equation*}
\lambda_p(d_I)=\inf_{\scriptstyle \varphi\in L^2(\Omega) \atop\scriptstyle
\varphi\neq0}\frac{\frac{d_I}{2}\int_{\Omega}\int_{\Omega}
J(x-y)(\varphi(y)-\varphi(x))^2dydx+\int_{\Omega}(\gamma(x)-\beta(x))
\varphi^2(x)dx}{\int_{\Omega}\varphi^2(x)dx}.
\end{equation*}
It is well-known that $\lambda_p(d_I)$ can be the unique principal eigenvalue of \eqref{201}, see \cite{Coville2010,HMMV2003,SLY2014,SHZH2010}.
\begin{lemma}\label{lemma207}
Set $m(x)=-d_I\int_{\Omega}J(x-y)dy+\beta(x)-\gamma(x)$. Suppose there is some $x_0\in Int(\Omega)$ satisfying that $m(x_0)=\max_{\bar{\Omega}}m(x)$, and the partial derivatives of $m(x)$ up to order $N-1$ at $x_0$ are zero. Then, $\lambda_{p }(d_I)$ is the unique principal eigenvalue of \eqref{201} and its corresponding eigenfunction $\varphi$ is positive and continuous on $\bar{\Omega}$.
\end{lemma}

\begin{lemma}\label{newlem}
$\lambda_p(d_I)$ is the principal eigenvalue of \eqref{201} if and only if $$\lambda_p(d_I)<\min_{\bar{\Omega}}
\left\{d_I\int_{\Omega}J(x-y)dy+\gamma(x)-\beta(x)\right\}.$$
\end{lemma}

The proof of Lemma \ref{newlem} is the same as Proposition 3.2 in \cite{Coville-2013}.

\begin{remark}\label{re208}
Note that $\lambda_p(d_I)$ is continuous on $J$, $\beta(x)$ and $\gamma(x)$, see the proof in \cite{Coville2010}.
\end{remark}

\begin{theorem}\label{theorem2203}
Assume $\lambda_{p }(d_I)$ is the principal eigenvalue of \eqref{201}. Then the following alternatives hold:
\begin{description}
\item[(i)] $\lambda_{p }(d_I)$ is a strictly monotone increasing function of $d_I$;
\item[(ii)] $\lambda_{p }(d_I)\to\min\limits_{\bar{\Omega}}\{\gamma(x)-\beta(x)\}$ as $d_I\to0$;
\item[(iii)] $\lambda_{p }(d_I)\to\frac{1}{|\Omega|}\int_{\Omega}(\gamma(x)-\beta(x))dx$ as $d_I\to+\infty$;
\item[(iv)] If $\int_{\Omega}\beta(x)dx\ge\int_{\Omega}\gamma(x)dx$, then $\lambda_{p }(d_I)<0$ for all $d_I>0$;
\item[(v)] If $\beta(x_*)>\gamma(x_*)$ for some $x_*\in\Omega$ and $\int_{\Omega}\beta(x)dx<\int_{\Omega}\gamma(x)dx$, then the equation $\lambda_{p }(d_I)=0$ has a unique positive root denoted by $d_{I}^{*}$. Furthermore, if $d_I<d_I^*$, then $\lambda_{p }(d_I)<0$ and if $d_I>d_I^*$, then $\lambda_{p }(d_I)>0$.
\end{description}
\end{theorem}
\begin{proof}
Let $\varphi(x)$ be the corresponding eigenfunction to $\lambda_p(d_I)$ and normalize it as $\|\varphi\|_{L^2(\Omega)}=1$. Then
\begin{equation*}
\lambda_p(d_I)=\frac{d_I}{2}\int_{\Omega}\int_{\Omega}
J(x-y)(\varphi(y)-\varphi(x))^2dydx+\int_{\Omega}
(\gamma(x)-\beta(x))\varphi^2(x)dx.
\end{equation*}
Obviously, $\varphi(x)$ is not a constant. Otherwise, $\lambda_p(d_I)=\gamma(x)-\beta(x)$ according to \eqref{201}, which is a contradiction. Assume $d_I>d_{I_1}$. Then, following the variational characterization of $\lambda_p(d_I)$, we have
\begin{equation*}
\lambda_p(d_I)>\frac{d_{I_1}}{2}\int_{\Omega}\int_{\Omega}
J(x-y)(\varphi(y)-\varphi(x))^2dydx+\int_{\Omega}
(\gamma(x)-\beta(x))\varphi^2(x)dx\ge\lambda_p(d_{I_1}).
\end{equation*}
This ends the proof of (i).

Now, let $\theta(x)=\gamma(x)-\beta(x)$ and $\theta_{min}=\min\limits_{\bar{\Omega}}\theta(x)$. Consider the eigenvalue problem
\begin{equation}\label{205}
d_I\int_{\Omega}
J(x-y)(u(y)-u(x))dy-\theta_{min}u(x)=-\lambda u(x)~\text{in}~\Omega.
\end{equation}
Thus, the principal eigenvalue of \eqref{205} is $\lambda_p^*=\theta_{min}$, see \cite{MR2009na}. Hence, we have $\lambda_p(d_I)\ge\theta_{min}$. Now, if we can prove that $\limsup\limits_{d_I\to0}\lambda_p(d_I)\le\theta_{min}$, then the result is obtained. On the contrary, assume there exists some $\varepsilon>0$ such that
$$\limsup_{d_I\to0}\lambda_p(d_I)\ge\theta_{min}+\varepsilon.$$
By the definition of $\limsup$, there exists some $\hat{d}_I>0$ such that if $d_I\le\hat{d}_I$, then
$$\lambda_p(d_I)\ge\theta_{min}+\frac\varepsilon 2.$$
Additionally, the continuity of $\theta(x)$ gives that there are some $x_0\in\Omega$ and $r>0$ such that
\begin{equation*}
\theta_{min}\ge \theta(x)-\frac\varepsilon 4~~\text{for}~x\in B_r(x_0)\subset\Omega.
\end{equation*}
Hence, $\lambda_p(d_I)\ge\theta(x)+\frac\varepsilon 4$ for $d_I\le\hat{d}_I$ and $x\in B_r(x_0)$. Let $\varphi(x)$ be the eigenfunction to $\lambda_p(d_I)$. Then, it follows from \eqref{201} that
\begin{equation*}
\int_{\Omega}J(x-y)(\varphi(x)-\varphi(y))dy
=\frac{\lambda_p(d_I)-\theta(x)}{d_I}\varphi(x)
\ge\frac{\varepsilon}{4d_I}\varphi(x)~~\text{in}~~B_r(x_0).
\end{equation*}
Let $\lambda_1$ be the principal eigenvalue of the linear problem
\begin{equation*}
\begin{cases}
\int_{\mathbb{R}^N}J(x-y)(u(y)-u(x))dy=-\lambda u(x)~&\text{in}~B_r(x_0),\\
u(x)=0~&\text{in}~\mathbb{R}^N\backslash B_r(x_0).
\end{cases}
\end{equation*}
It is well-known that $0<\lambda_1<1$, see \cite{MR2009}. Let $\psi(x)$ be the corresponding eigenfunction to $\lambda_1$ normalized by $\|\psi\|_{L^\infty(B_r(x_0))}=1$. Set
\begin{equation*}
\overline{\Phi}(x)=\frac{\varphi(x)}{\inf_{B_r(x_0)}\varphi(x)},~~~
\underline{\Phi}(x)=\psi(x)\le1.
\end{equation*}
Consider the following linear problem
\begin{equation}\label{206}
\begin{cases}
\int_{\mathbb{R}^N}J(x-y)(u(y)-u(x))dy=-\frac{\varepsilon}{4d_I} u(x)~&\text{in}~B_r(x_0),\\
u(x)=0~&\text{in}~\mathbb{R}^N\backslash B_r(x_0).
\end{cases}
\end{equation}
By the direct computation, we have
\begin{equation*}
\begin{aligned}
&\int_{B_r(x_0)}J(x-y)\overline{\Phi}(y)dy-\overline{\Phi}(x)
+\frac{\varepsilon}{4d_I}\overline{\Phi}(x)\\
\le& \int_{\Omega}J(x-y)(\overline{\Phi}(y)-\overline{\Phi}(x))dy
+\frac{\varepsilon}{4d_I}\overline{\Phi}(x)\le0
\end{aligned}
\end{equation*}
and
\begin{equation*}
\begin{aligned}
&\int_{B_r(x_0)}J(x-y)\underline{\Phi}(y)dy-\underline{\Phi}(x)
+\frac{\varepsilon}{4d_I}\underline{\Phi}(x)\\
=& \int_{B_r(x_0)}J(x-y)\psi(y)dy-\psi(x)
+\frac{\varepsilon}{4d_I}\psi(x)\ge0
\end{aligned}
\end{equation*}
provided $d_I\le\min\{\hat{d}_I, \frac{\varepsilon}{4\lambda_1}\}$. Then, by the super-sub solution method, \eqref{206} admits a positive solution between $\underline{\Phi}(x)$ and $\overline{\Phi}(x)$, which implies that $\lambda_1=\frac{\varepsilon}{4d_I}$. This contradicts to the independence of $d_I$ about $\lambda_1$. Thus, $\lim\limits_{d_I\to0}\lambda_p(d_I)=\theta_{min}$. This completes the proof of (ii).

Now, taking $\varphi^2=\frac{1}{|\Omega|}$, the definition of $\lambda_p(d_I)$ yields that
\begin{equation*}
\begin{aligned}
\lambda_p(d_I)&\le\frac{\frac{d_I}{2}\int_{\Omega}\int_{\Omega}
J(x-y)(\varphi(y)-\varphi(x))^2dydx+\int_{\Omega}(\gamma(x)-\beta(x))
\varphi^2(x)dx}{\int_{\Omega}\varphi^2(x)dx}\\
&=\frac{1}{|\Omega|}\int_{\Omega}(\gamma(x)-\beta(x))dx\le
\max_{\bar{\Omega}}(\gamma(x)-\beta(x)).
\end{aligned}
\end{equation*}
Since $\lambda_p(d_I)$ is strictly increasing on $d_I$, the limit of $\lambda_p(d_I)$ exists as $d_I\to+\infty$. Assume $\lim\limits_{d_I\to+\infty}\lambda_p(d_I)=\lambda_\infty$. Then, $\lambda_\infty\le\max_{\bar{\Omega}}(\gamma(x)-\beta(x))$. Letting $\psi_{d_I}(x)$ be the corresponding eigenfunction to $\lambda_p(d_I)$ and normalizing it by $\|\psi_{d_I}\|_{L^\infty(\Omega)}=1$, then we have
\begin{equation}\label{207}
d_I\int_{\Omega}J(x-y)(\psi_{d_I}(y)-\psi_{d_I}(x))dy
+(\beta(x)-\gamma(x))\psi_{d_I}(x)=-\lambda_p(d_I)\psi_{d_I}(x)
~~\text{in}~\Omega.
\end{equation}
Note that there exists some $d_0>0$ such that
\begin{equation*}
\int_{\Omega}J(x-y)dy+\frac{\gamma(x)-\beta(x)-\lambda_p(d_I)}{d_I}>0
\end{equation*}
for any $d_I\ge d_0$. Then, for each $d_I\ge d_0$, \eqref{207} implies that
\begin{equation}\label{222}
\psi_{d_I}(x)=\frac{\int_{\Omega}J(x-y)\psi_{d_I}(y)dy}
{\int_{\Omega}J(x-y)dy+\frac{\gamma(x)-\beta(x)-\lambda_p(d_I)}{d_I}}\in C(\bar{\Omega}).
\end{equation}
Choose a sequence $\{d_{I,n}\}_{n=1}^{\infty}$ satisfying $d_{I,n}\to+\infty$ as $n\to+\infty$. Thus, the eigenfunction sequence $\{\psi_{d_{I,n}}\}$ weakly converges to some $\psi(x)$ in $L^2(\Omega)$. Hence, we have
\begin{equation*}
\int_{\Omega}J(x-y)\psi_{d_{I,n}}(y)dy\to
\int_{\Omega}J(x-y)\psi(y)dy~~~\text{as}~n\to+\infty.
\end{equation*}
Following \eqref{222}, there is
\begin{equation*}
\psi_{d_{I,n}}(x)\to\psi(x)~\text{uniformly on}~\bar{\Omega}~\text{as}~n\to+\infty.
\end{equation*}
Since
\begin{equation*}
\int_{\Omega}J(x-y)(\psi_{d_{I,n}}(y)-\psi_{d_{I,n}}(x))dy
=\frac{\gamma(x)-\beta(x)-\lambda_p(d_{I,n})}{d_{I,n}}\psi_{d_{I,n}}(x)~~\text{in}~\Omega,
\end{equation*}
we have
\begin{equation*}
\int_{\Omega}J(x-y)(\psi_{d_{I,n}}(y)-\psi_{d_{I,n}}(x))dy\to0~~\text{uniformly on}~\bar{\Omega}~\text{as}~n\to+\infty.
\end{equation*}
According to Proposition 3.3 in \cite{AMRT2010}, we know $\psi(x)$ is a constant. Integrating both sides of \eqref{207} over $\Omega$, we have
\begin{equation*}
\int_{\Omega}(\beta(x)-\gamma(x))\psi_{d_{I,n}}(x)dx
=-\lambda_p(d_{I,n})\int_{\Omega}\psi_{d_{I,n}}(x)dx.
\end{equation*}
Thus, there is
\begin{equation*}
\lim_{n\to+\infty}\lambda_p(d_{I,n})=\frac{1}{|\Omega|}
\int_{\Omega}(\gamma(x)-\beta(x))dx.
\end{equation*}
Additionally, by the definition of $\lambda_p(d_I)$, (iv) is obvious. Meanwhile, (v) is the direct conclusion of (i)-(iii). The proof is complete.
\end{proof}

Define an operator as follows
\begin{equation}\label{202}
A[u](x):=d_I\int_{\Omega}J(x-y)(u(y)-u(x))dy-\gamma(x)u(x).
\end{equation}
Then, we have the result as below.
\begin{proposition}\label{pro205}
If the operator $A$ is defined by \eqref{202}, then $A$ is a resolvent-positive operator on $X$ and $S(A)<0$.
\end{proposition}
\begin{proof}
By the definition of the operator $A$, we know $A$ is a bounded linear operator on $X$. It is known that the operator $A$ can generate a positive $C_0$-semigroup, see \cite{KLSH2010}. Then, following from Theorem \ref{theorem203}, we have $A$ is a resolvent-positive operator on $X$.

Let
\begin{equation*}
\lambda_p:=\sup_{\scriptstyle \varphi\in L^2(\Omega) \atop\scriptstyle
\varphi\neq0}\frac{-\frac{d_I}{2}\int_{\Omega}\int_{\Omega}
J(x-y)(\varphi(y)-\varphi(x))^2dydx-\int_{\Omega}\gamma(x)
\varphi^2(x)dx}{\int_{\Omega}\varphi^2(x)dx}.
\end{equation*}
Obviously, $\lambda_p<0$. Let $h(x)=-d_I\int_{\Omega}
J(x-y)dy-\gamma(x)$. We may choose some function sequence $\{h_n(x)\}_{n=1}^{\infty}$ with $\|h_n-h\|_{L^\infty(\Omega)}\to0$ as $n\to+\infty$ such that the eigenvalue problem
\begin{equation*}
A_n[\varphi](x):=d_I\int_{\Omega}
J(x-y)\varphi(y)dy+h_n(x)\varphi(x)=\lambda\varphi(x)~\text{in}~\Omega
\end{equation*}
admits a principal eigenpair denoted by $(\lambda_p^n, \varphi_n(x))$, where $\lambda_p^n\to\lambda_p$ as $n\to+\infty$. Note that $\lambda_p^n=S(A_n)$ for each given $n$ (see Bates and Zhao \cite{BZH2007}). Since $\lambda_p<0$, there exists some $\delta>0$ such that $\lambda_p^n<-\delta$ provided $n\ge n_0$ for some $n_0>0$. Thus, we have $S(A_n)<-\delta$ for $n\ge n_0$. Due to $h_n\to h$ as $n\to+\infty$, we can obtain that $S(A_n)\to S(A)$ as $n\to+\infty$, see Lemma 3.1 in \cite{SZ2012}. This implies that $S(A)<0$. The proof is complete.

\end{proof}

Now, consider the nonlocal dispersal problem
\begin{equation}\label{2011}
\frac{\partial u_I(x,t)}{\partial t}=d_I\int_{\Omega}J(x-y)(u_I(y,t)-u_I(x,t))dy-\gamma(x)u_I(x,t),
\end{equation}
where $x\in\Omega$ and $t>0$. If $u_I(x,t)$ is thought of as a density of the infected individuals at a point $x$ at time $t$, $J(x-y)$ is thought of as the probability distribution of jumping from location $y$ to location $x$, then $\int_{\Omega}J(y-x)u(y,t)dy$ is the rate at which the infected individuals are arriving at position $x$ from all other places, and $-\int_{\Omega}J(y-x)u(x,t)dy$ is the rate at which they are leaving location $x$ to travel to all other sites.
By the theory of semigroups of linear operators, we know that the operator $A$ can generate a uniformly continuous semigroup, denoted by $T(t)$. Suppose that $\phi(x)$ is the distribution of initial infection at location $x$. Then the distribution of those infective members is at time t (as time evolution) is $\left(T(t)\phi\right)(x)$.
Set $\mathscr{F}[\varphi](x):=\beta(x)\varphi(x)$ for $\varphi\in X$. Hence, the distribution of new infection at time $t$ is $\mathscr{F}[T(t)\phi](x)$ and the total new infections are
\begin{equation*}
\int_{0}^{\infty}\mathscr{F}[T(t)\phi](x)dt.
\end{equation*}
Define
\begin{equation*}
L[\phi](x):=\int_{0}^{\infty}\mathscr{F}[T(t)\phi](x)dt=
\beta(x)\int_{0}^{\infty}T(t)\phi dt.
\end{equation*}
Then, inspired by the ideas of next generation operators (see \cite{WZ2012,Diekmann1990,Driessche2002,WZ2011}), we may define the spectral radius of $L$ as the basic reproduction number of system \eqref{101}, that is $R_0=r(L)$.
After then, we have the following result.
\begin{theorem}\label{theorem206}
$R_0-1$ has the same sign as $\lambda_*:=S(A+\mathscr{F})$.
\end{theorem}
\begin{proof}
Since $A$ is the generator of the semigroup $T(t)$ on $X$ and $A$ is resolvent-positive, it then follows from Theorem \ref{theorem203} that
\begin{equation}\label{203}
(\lambda I-A)^{-1}\phi=\int_{0}^{\infty}e^{-\lambda t}T(t)\phi dt~~\text{for any}~\lambda>S(A),~\phi\in X.
\end{equation}
Choosing $\lambda=0$ in \eqref{203}, we obtain
\begin{equation}\label{2008}
-A^{-1}\phi=\int_{0}^{\infty}T(t)\phi dt~\text{for all}~\phi\in X.
\end{equation}
Then, the definition of the operator $L$ implies that $L=-\mathscr{F}A^{-1}$. Let $\mathcal{M}:=A+\mathscr{F}$.
We know that $\mathcal{M}$ can generate a uniformly continuous positive semigroup, then $\mathcal{M}$ is resolvent-positive. Meanwhile, $S(A)<0$. Thus, following from Theorem \ref{theorem204}, we have $S(\mathcal{M})$ has the same sign as $r(-\mathscr{F}A^{-1})-1=R_0-1$. The proof is complete.
\end{proof}

Note that if $\lambda_p(d_I)$ is the principal eigenvalue of \eqref{201}, then $-\lambda_p(d_I)=S(A+\mathscr{F})$. In this case, $-\lambda_p(d_I)$ has the same sign as $R_0-1$ according to Theorem \ref{theorem206}. However, we still have the following result whether $\lambda_p(d_I)$ is the principal eigenvalue of \eqref{201} or not.
\begin{corollary}\label{remark209}
$\lambda_p(d_I)$ has the same sign as $1-R_0$.
\end{corollary}

In fact, this is easily seen from the proof in Proposition \ref{pro205} that $-\lambda_p(d_I)=S(A+\mathscr{F})$. Thus, Corollary \ref{remark209} is obvious.

\begin{corollary}\label{newcor1}
If $\beta(x_0)>\gamma(x_0)$ for some $x_0\in\Omega$ and $\int_{\Omega}\beta(x)dx<\int_{\Omega}\gamma(x)dx$. Then there exists some $d_*>0$ such that $R_0>1$ for all $0<d_I<d_*$ and $R_0<1$ for $d_I>d_*$.
\end{corollary}
\begin{proof}
\,Since $\beta(x_0)>\gamma(x_0)$ for some $x_0\in\Omega$, the continuity of $\beta(x)$ and $\gamma(x)$ gives that $\beta(x)>\gamma(x)$ for any $x\in B_r(x_0)$, which $B_r(x_0)$ is a ball with the center $x_0$ and the radius $r>0$. Let $\Omega_*=B_r(x_0)\cap\Omega$ and denote
\begin{equation*}
\tilde{\varphi}:=
\begin{cases}
C,&~~x\in \Omega_*,\\
0,&~~x\in \Omega\backslash \Omega_*,
\end{cases}
\end{equation*}
for some nonzero constant $C$. Then, by the definition of $\lambda_p(d_I)$ and the continuity of $\lambda_p(d_I)$ on $d_I$, we have
$$\lambda_p(0)<\int_{\Omega_*}(\gamma(x)-\beta(x))dx<0.$$

Moreover, it follows from the definition of $\lambda_p(d_I)$ that
\begin{equation*}
\lambda_p(d_I)\le\max_{\bar{\Omega}}\{\gamma(x)-\beta(x)\}.
\end{equation*}
Then, there exists some $\hat{d}>0$ such that
\begin{equation*}
\lambda_p(d_I)<\max_{\bar{\Omega}}\left\{
d_I\int_{\Omega}J(x-y)dy+\gamma(x)-\beta(x)\right\}
\end{equation*}
for any $d_I>\hat{d}$. According to Lemma \ref{newlem}, $\lambda_p(d_I)$ is the principal eigenvalue of \eqref{201} for $d_I>\hat{d}$. Thus, using Theorem \ref{theorem2203}, we have
\begin{equation*}
\lim_{d_I\to+\infty}\lambda_p(d_I)=\frac{1}{|\Omega|}
\int_{\Omega}(\gamma(x)-\beta(x))dx.
\end{equation*}
Since $\lambda_p(d_I)$ is nondecreasing on $d_I$, there is some $d_*>0$ such that
\begin{equation*}
\lambda_p(d_I)
\begin{cases}
<0~~\text{if}~~0<d_I<d_*,\\
>0~~\text{if}~~d_I>d_*.
\end{cases}
\end{equation*}
Thus, using Corollary \ref{remark209}, we can finish our proof.
\end{proof}

\begin{corollary}\label{newcor}
If $\int_{\Omega}\beta(x)dx>\int_{\Omega}\gamma(x)dx$, then $R_0>1$ for any $d_I>0$. Further, if $\beta(x)<\gamma(x)$ for $x\in\Omega$, then $R_0<1$ for all $d_I>0$.
\end{corollary}

This is easy seen from the definition of $\lambda_p(d_I)$ and Corollary \ref{remark209}.

\begin{lemma}\label{lemma211}
Assume $(\mu_p, \phi(x))$ with $\phi(x)>0$ is a principal eigenpair of the weighted eigenvalue problem
\begin{equation}\label{2009}
-d_I\int_{\Omega}J(x-y)(\phi(y)-\phi(x))dy+\gamma(x)\phi(x)=
\mu\beta(x)\phi(x),~~x\in\Omega.
\end{equation}
Then, $\mu_p$ is a unique positive principal eigenvalue and can be characterized by
\begin{equation*}
\mu_p=\inf_{\scriptstyle \varphi\in L^2(\Omega) \atop\scriptstyle
\varphi\neq0}\frac{\frac{d_I}{2}\int_{\Omega}\int_{\Omega}
J(x-y)(\varphi(y)-\varphi(x))^2dydx+\int_{\Omega}\gamma(x)
\varphi^2(x)dx}{\int_{\Omega}\beta(x)\varphi^2(x)dx}.
\end{equation*}
\end{lemma}
\begin{proof}
Let $(\mu_i, \phi_i(x))$ $(i=1,2)$ with $\phi_i(x)>0$ satisfying
\begin{equation*}
-d_I\int_{\Omega}J(x-y)(\phi_i(y)-\phi_i(x))dy+\gamma(x)\phi_i(x)=
\mu_i\beta(x)\phi_i(x).
\end{equation*}
Following these equations, it is easy to obtain that
\begin{equation*}
(\mu_1-\mu_2)\int_{\Omega}\beta(x)\phi_1(x)\phi_2(x)dx=0.
\end{equation*}
The positivity of $\phi_1(x)$ and $\phi_2(x)$ gives that $\mu_1=\mu_2$. Further, according to \eqref{2009}, there is
\begin{equation}\label{2010}
\mu_p=\frac{\frac{d_I}{2}\int_{\Omega}\int_{\Omega}
J(x-y)(\phi(y)-\phi(x))^2dydx+\int_{\Omega}\gamma(x)
\phi^2(x)dx}{\int_{\Omega}\beta(x)\phi^2(x)dx}.
\end{equation}
Obviously, $\mu_p>0$.

Below, we prove that
\begin{equation*}
\mu_p=\mu_p^\prime:=\inf_{\scriptstyle \varphi\in L^2(\Omega) \atop\scriptstyle
\varphi\neq0}\frac{\frac{d_I}{2}\int_{\Omega}\int_{\Omega}
J(x-y)(\varphi(y)-\varphi(x))^2dydx+\int_{\Omega}\gamma(x)
\varphi^2(x)dx}{\int_{\Omega}\beta(x)\varphi^2(x)dx}.
\end{equation*}
In view of \eqref{2010}, we have $\mu_p\ge \mu_p^\prime$. Assume that $\mu_p>\mu_p^\prime$. Then, there exists some $\mu_*$ such that $\mu_p^\prime<\mu_*<\mu_p$. Set
$$\mathbb{H}(\varphi)=\frac{d_I}{2}\int_{\Omega}\int_{\Omega}
J(x-y)(\varphi(y)-\varphi(x))^2dydx+\int_{\Omega}\gamma(x)
\varphi^2(x)dx$$
and define
\begin{equation*}
\sigma(\mu)=\sup_{\scriptstyle \varphi\in L^2(\Omega) \atop\scriptstyle
\varphi\neq0}\frac{\mu\int_{\Omega}\beta(x)\varphi^2(x)dx
-\mathbb{H}(\varphi)}{\int_{\Omega}\varphi^2(x)dx}.
\end{equation*}
Then, it follows from \eqref{2010} that $\sigma(\mu_p)=0$. Since $\mu_*>\mu_p^\prime$, there is some $v\in L^2(\Omega)$ and $v\neq0$ satisfying
\begin{equation*}
\mu_*>\frac{\frac{d_I}{2}\int_{\Omega}\int_{\Omega}
J(x-y)(v(y)-v(x))^2dydx+\int_{\Omega}\gamma(x)
v^2(x)dx}{\int_{\Omega}\beta(x)v^2(x)dx}>0.
\end{equation*}
This implies that $\sigma(\mu_*)>0$. On the other hand, by the definition of $\sigma(\mu)$, it is easy to see that $\sigma(\mu)$ is nondecreasing on $\mu$. Due to $\mu_*<\mu_p$, we have $\sigma(\mu_*)\le\sigma(\mu_p)$. That is $\sigma(\mu_*)\le0$, which is a contradiction. We end the proof.
\end{proof}

\begin{corollary}\label{cor2112}
If $(\mu^*, \phi^*(x))$ with $\phi^*(x)>0$ satisfies the following linear problem
\begin{equation*}
\begin{cases}
\int_{\mathbb{R}^N}J(x-y)(\phi^*(y)-\phi^*(x))dy=-\mu\gamma(x) \phi^*(x)&~~~\text{in}~\Omega,\\
\phi^*(x)=0&~~~\text{on}~\mathbb{R}^N\backslash\Omega,
\end{cases}
\end{equation*}
then $\mu^*$ is unique and positive.
\end{corollary}

\begin{lemma}\label{theorem212}
If the nonlocal weighted eigenvalue problem
\begin{equation*}
-d_I\int_{\Omega}J(x-y)(\phi(y)-\phi(x))dy+\gamma(x)\phi(x)=
\mu\beta(x)\phi(x),~~x\in\Omega
\end{equation*}
admits a unique positive principal eigenvalue $\mu_p$ with positive eigenfunction and there exists some positive function $\psi_{d_I}(x)\in L^2(\Omega)$ satisfying
$$L[\psi_{d_I}](x)=R_0\psi_{d_I}(x),$$
then $R_0=r(-\mathscr{F}A^{-1})=\frac{1}{\mu_p}$ and the following two conclusions hold:
\begin{description}
\item[(i)] $R_0\to\max\limits_{\bar{\Omega}}\{\frac{\beta(x)}{\gamma(x)}\}$ as $d_I\to0$;
\item[(ii)] $R_0\to\frac{\int_{\Omega}\beta(x)dx}{\int_{\Omega}\gamma(x)dx}$ as $d_I\to+\infty$.
\end{description}
\end{lemma}
\begin{proof}
Note that
\begin{equation*}
\beta(x)\int_{0}^{\infty}T(t)\psi_{d_I} dt=R_0\psi_{d_I}(x).
\end{equation*}
In view of \eqref{2008}, we have $-A^{-1}\psi_{d_I}=\int_{0}^{\infty}T(t)\psi_{d_I} dt$. Accordingly,
\begin{equation}\label{2012}
-\beta(x)A^{-1}[\psi_{d_I}](x)=R_0\psi_{d_I}(x).
\end{equation}
Let $\varphi=-A^{-1}\psi_{d_I}$. Obviously, $\varphi$ is positive. It follows from \eqref{2012} that $-A\varphi=\frac{1}{R_0}\beta(x)\varphi$. That is $(\frac{1}{R_0}, \varphi)$ satisfies
\begin{equation*}
-d_I\int_{\Omega}J(x-y)(\varphi(y)-\varphi(x))dy+\gamma(x)\varphi(x)=
\frac{1}{R_0}\beta(x)\varphi(x).
\end{equation*}
Following Lemma \ref{lemma211}, it is clear that $(\frac{1}{R_0}, \varphi)$ is the principal eigenpair of \eqref{2009}. Hence, $R_0=r(-\mathscr{F}A^{-1})=\frac{1}{\mu_p}$.
Meanwhile, $R_0$ can be characterized by
\begin{equation}\label{star}
R_0=\sup_{\scriptstyle \varphi\in L^2(\Omega) \atop\scriptstyle
\varphi\neq0}\frac{\int_{\Omega}\beta(x)\varphi^2(x)dx}{\frac{d_I}{2}\int_{\Omega}\int_{\Omega}
J(x-y)(\varphi(y)-\varphi(x))^2dydx+\int_{\Omega}\gamma(x)
\varphi^2(x)dx}.
\end{equation}

Now, we prove (i) and (ii).
Denote $R_0=R_0(d_I)$ and $\eta(x)=\frac{\beta(x)}{\gamma(x)}$. For any $v\in L^2(\Omega)$ and $v\neq0$, we have
\begin{eqnarray*}
&& \frac{\int_{\Omega}\beta(x)v^2(x)dx}{\frac{d_I}{2}
\int_{\Omega}\int_{\Omega}
J(x-y)(v(y)-v(x))^2dydx+\int_{\Omega}\gamma(x)
v^2(x)dx}\\
&\le&
\frac{\max_{\bar{\Omega}}
\eta(x)\int_{\Omega}\gamma(x)v^2(x)dx}{\frac{d_I}{2}
\int_{\Omega}\int_{\Omega}
J(x-y)(v(y)-v(x))^2dydx+\int_{\Omega}\gamma(x)
v^2(x)dx}\\
&\le& \max_{\bar{\Omega}}
\eta(x).
\end{eqnarray*}
Hence, $R_0(d_I)\le\max_{\bar{\Omega}}
\eta(x):=\eta_*$. To our goal, we only need to prove that $\liminf\limits_{d_I\to0}R_0(d_I)\ge\eta_*$. On the contrary, assume there exists some $\varepsilon>0$ such that
$$\liminf_{d_I\to0}R_0(d_I)\le\eta_*-\varepsilon.$$
By the definition of liminf, there is some $d_0>0$ such that
$$R_0(d_I)\le\eta_*-\frac\varepsilon2$$
for any $d_I\le d_0$. Additionally,
the continuity of $\eta(x)$ gives that there exists some $x_*\in\bar{\Omega}$ so that
\begin{equation*}
\eta_*\le\eta(x)+\frac\varepsilon4~\text{for any}~x\in B_\rho(x_*),
\end{equation*}
in which $B_\rho(x_*)$ is a ball with the center $x_*$ and the radius $\rho$. Hence,
\begin{equation*}
R_0(d_I)\le\eta(x)-\frac\varepsilon4~~~\text{for all}~x\in B_\rho(x_*).
\end{equation*}
It is noticed that
\begin{eqnarray*}
d_I\int_{\Omega}J(x-y)(\psi_{d_I}(x)-\psi_{d_I}(y))dy &=& \left(\frac{\beta(x)}{R_0(d_I)}-\gamma(x)\right)\psi_{d_I}(x)\\
&\ge& \left(\frac{\beta(x)}{\eta(x)-\frac\varepsilon4}
-\gamma(x)\right)\psi_{d_I}(x)\\
&=& \frac{\varepsilon\gamma(x)}
{4(\eta(x)-\frac\varepsilon4)}\psi_{d_I}(x)\\
&\ge& \frac{\varepsilon\gamma(x)}{4\eta_*}\psi_{d_I}(x).
\end{eqnarray*}
On the other hand, it follows from \cite{MR2009} that the problem
\begin{equation*}
\begin{cases}
\int_{\mathbb{R}^N}J(x-y)(v(y)-v(x))dy
=-\mu\max_{\bar{\Omega}}\{\gamma(x)\}v(x)~~&\text{in}~B_\rho(x_*),\\
v(x)=0~~&\text{on}~\mathbb{R}^N\backslash B_\rho(x_*)
\end{cases}
\end{equation*}
admits a principal eigenpair $(\tilde{\mu}, \varphi^*(x))$ and $0<\tilde{\mu}<\frac{1}{\max_{\bar{\Omega}}\gamma(x)}$. Now, let
\begin{equation*}
\underline{\Psi}(x)=\frac{\varphi^*(x)}
{\inf_{B_\rho(x_*)}\varphi^*(x)},~~\overline{\Psi}(x)=K\psi_{d_I}(x)
~~\text{for constant}~K>1.
\end{equation*}
For the simple calculation, $\underline{\Psi}(x)$ and $\overline{\Psi}(x)$ are a pair of sub-super solution of the following linear problem
\begin{equation}\label{2201}
\begin{cases}
\int_{\mathbb{R}^N}J(x-y)(u(y)-u(x))dy
=-\frac{\varepsilon\gamma(x)}{4d_I\eta_*}u(x)~~&\text{in}~B_\rho(x_*),\\
u(x)=0~~&\text{on}~\mathbb{R}^N\backslash B_\rho(x_*)
\end{cases}
\end{equation}
when $d_I\le\min\{d_0,\frac{\varepsilon}{4\tilde{\mu}\eta_*}\}$. Then, there is a positive solution of \eqref{2201}. Following Corollary \ref{cor2112}, it is obtained that $\mu_p^\prime:=\frac{\varepsilon}{4d_I\eta_*}$ is a principal eigenvalue of \eqref{2201} which depends on the parameter $d_I$ and this is a contradiction.

Next, we prove (ii). By the variational characterization of $R_0(d_I)$, it is easily seen that
\begin{equation*}
R_0(d_I)\ge \frac{\int_{\Omega}\beta(x)dx}{\int_{\Omega}\gamma(x)dx}.
\end{equation*}
Since $R_0(d_I)$ is nondecreasing on $d_I$, the limit of $R_0(d_I)$ exists as $d_I\to+\infty$. Noticed that $(R_0(d_I), \psi_{d_I}(x))$ satisfies
\begin{equation}\label{2014}
-d_I\int_{\Omega}J(x-y)(\psi_{d_I}(y)-\psi_{d_I}(x))dy
+\gamma(x)\psi_{d_I}(x)=
\frac{1}{R_0(d_I)}\beta(x)\psi_{d_I}(x).
\end{equation}
Choose some sequence $\{d_{I,n}\}_n^\infty$ satisfying $d_{I,n}\to+\infty$ as $n\to+\infty$ and normalized $\psi_{d_{I,n}}(x)$ as $\|\psi_{d_{I,n}}\|_{L^\infty(\Omega)}=1$. Since there is some $n_0>0$ such that
\begin{equation*}
\Delta(x):=\int_{\Omega}J(x-y)dy+\frac{\gamma(x)
-\frac{\beta(x)}{R_0(d_{I,n})}}{d_{I,n}}>0
\end{equation*}
for all $n\ge n_0$, we have
\begin{equation*}
\psi_{d_{I,n}}(x)=\frac{\int_{\Omega}J(x-y)\psi_{d_{I,n}}(y)dy}
{\Delta(x)}
\end{equation*}
for all $n\ge n_0$. Thus, $\psi_{d_{I,n}}(x)\to\psi^*$ strongly in $L^2(\Omega)$ as $n\to+\infty$. This implies that $\psi^*$ satisfies
\begin{equation*}
\int_{\Omega}J(x-y)(\psi^*(y)-\psi^*(x))dy=0.
\end{equation*}
Hence, $\psi^*$ is a positive constant. Integrating both sides of \eqref{2014} with $d_{I,n}$ and $\psi_{d_{I,n}}(x)$ on $\Omega$ obtains that
\begin{equation*}
\int_{\Omega}\gamma(x)\psi_{d_{I,n}}(x)dx
=\frac{1}{R_0(d_{I,n})}\int_{\Omega}\beta(x)\psi_{d_{I,n}}(x)dx.
\end{equation*}
Letting $n\to+\infty$, there holds
\begin{equation*}
\lim_{n\to+\infty}R_0(d_{I,n})=\frac{\int_{\Omega}\beta(x)dx}
{\int_{\Omega}\gamma(x)dx}.
\end{equation*}
This ends the proof.
\end{proof}

\begin{remark}{\rm
Comparing to the corresponding elliptic problem, the operator
\begin{equation*}
L[\varphi](x)=\beta(x)\int_{0}^{\infty}T(t)\varphi dt
\end{equation*}
is not a compact operator. Thus, $r(L)$ may not be a principal eigenvalue of $L$ and the basic reproduction number $R_0$ can not be characterized as \eqref{star} in general. However, \eqref{star} is still able to discuss the dynamic behavior of system \eqref{101} as a threshold value.
}
\end{remark}

\section{The disease-free equilibrium}\label{sec3}
\noindent

In this section, we major in discussing the existence and stability of the disease-free equilibrium of \eqref{101}. That is, we consider the stationary problem of system \eqref{101}:
\begin{equation}\label{301}
\begin{cases}
d_S\int_{\Omega}J(x-y)(S(y)-S(x))dy
=\frac{\beta(x)SI}{S+I}-\gamma(x)I,~& x\in\Omega,\\
d_I\int_{\Omega}J(x-y)(I(y)-I(x))dy
=-\frac{\beta(x)SI}{S+I}+\gamma(x)I,~& x\in\Omega.
\end{cases}
\end{equation}

\begin{definition}\label{def301}
We say that a steady state $(\tilde{S}(x),\tilde{I}(x))$ of system \eqref{101} is globally stable if the solutions $(S_*(x,t),I_*(x,t))$ of \eqref{101} satisfy
\begin{equation*}
\lim_{t\to+\infty}(S_*(x,t),I_*(x,t))=(\tilde{S}(x),\tilde{I}(x))
\end{equation*}
for any initial data $(S_0(x),I_0(x))$ that satisfies $S_0(x),I_0(x)>0$ in $\Omega$ and $S_0(x),I_0(x)\in C(\bar{\Omega})$.
\end{definition}

\begin{lemma}\label{lemma301}
System \eqref{301} admits a disease-free equilibrium $(\hat{S},0)$, it is unique and given by $\hat{S}=\frac{N}{|\Omega|}$ on $\bar{\Omega}$.
\end{lemma}
\begin{proof}
Let $(\tilde{S},0)$ be any disease-free equilibrium. Then, following \eqref{301}, we obtain that
\begin{equation*}
\int_{\Omega}J(x-y)(\tilde{S}(y)-\tilde{S}(x))dy=0~~\text{in}~\Omega.
\end{equation*}
It is well-known from \cite[Proposition 3.3]{AMRT2010} that $\tilde{S}(x)$ is a constant. And since $\int_{\Omega}\tilde{S}(x)dx=N$, we have $\tilde{S}(x)=\frac{N}{|\Omega|}$ on $\bar{\Omega}$. The proof is complete.
\end{proof}

Then, we have the following globally stability result.
\begin{theorem}\label{theorem302}
If $R_0<1$, then all the positive solutions of \eqref{101} converge to the disease-free equilibrium $\left(\frac{N}{|\Omega|}, 0\right)$ as $t\to+\infty$.
\end{theorem}
\begin{proof}
Since $R_0<1$, we have $-\lambda_p(d_I)=\lambda_*<0$ according to Corollary \ref{remark209}. That is $\lambda_p(d_I)>0$. Recall that $m(x)=-d_I\int_{\Omega}J(x-y)dy+\beta(x)-\gamma(x)$. Moreover, since $m(x)$ is continuous on $\bar{\Omega}$, there exists some $x_0\in\bar{\Omega}$ such that $m(x_0)=\max\limits_{x\in\bar{\Omega}}m(x)$. Define a function sequence as follows:
\begin{equation*}
m_n(x)=
\begin{cases}
m(x_0),~& x\in B_{x_0}(\frac1n),\\
m_{n,1}(x),~& x\in(B_{x_0}(\frac2n)\backslash B_{x_0}(\frac1n)),\\
m(x),~& x\in\Omega\backslash B_{x_0}(\frac2n),
\end{cases}
\end{equation*}
where $B_{x_0}(\frac1n)=\{x\in\Omega|~|x-x_0|<\frac1n\}$, $m_{n,1}(x)$ satisfies $m_{n,1}\le m(x_0)$, and $m_{n,1}(x)$ is continuous in $\Omega$. Indeed, $m_{n,1}(x)$ exists if only we take $n$ is large enough, denoted by $n\ge n_0>0$. Thus, Lemma \ref{lemma207} implies that the eigenvalue problem
\begin{equation*}
d_I\int_{\Omega}J(x-y)\phi(y)dy+m_n(x)\phi(x)=-\lambda\phi(x)
\end{equation*}
admits a principal eigenpair, denoted by $(\lambda_p^n(d_I),\phi_n)$. According to Remark \ref{re208}, there exists some $n_1\ge n_0$ such that for any $n\ge n_1$
\[
\lambda_p^n(d_I)\ge\frac12\lambda_p(d_I)
-\|m_n-m\|_{L^\infty}.
\]
Normalizing $\phi_n(x)$ as $\|\phi_n\|_{L^\infty(\Omega)}=1$ and letting $\overline{u}(x,t)=Me^{-\frac12\lambda_p(d_I)t}\phi_n(x)$, the direct calculation yields that
\begin{eqnarray*}
&&\frac{\partial\overline{u}(x,t)}{\partial t}-d_I\int_{\Omega}J(x-y)(\overline{u}(y,t)-\overline{u}(x,t))dy
-\frac{\beta(x)\overline{u}S_*}{\overline{u}+S_*}+\gamma(x)\overline{u}\\
&\ge& -\frac12\lambda_p(d_I)Me^{-\frac12\lambda_p(d_I)t}\phi_n(x)
-Me^{-\frac12\lambda_p(d_I)t}
\left[d_I\int_{\Omega}J(x-y)\phi_n(y)dy+m_n(x)\phi_n(x)\right]\\
&&+(m_n(x)-m(x))Me^{-\frac12\lambda_p(d_I)t}\phi_n(x)\\
&\ge& \left[ \lambda_p^n(d_I)-\frac12\lambda_p(d_I)
+(m_n(x)-m(x)) \right ]Me^{-\frac12\lambda_p(d_I)t}\phi_n(x)\ge0,
\end{eqnarray*}
provided $n\ge n_1$. Take $M$ large enough such that $\overline{u}(x,0)\ge I_0(x)$. Then, the comparison principle yields that $I_*(x,t)\le \overline{u}(x,t)$ for $x\in\Omega$ and $t>0$. Consequently, we get that $I_*(x,t)\to0$ uniformly on $\bar{\Omega}$ as $t\to+\infty$.

Now, it is left to prove that
$S_*(x,t)\to\frac{N}{|\Omega|}$ uniformly on $\bar{\Omega}$ as $t\to+\infty$.
By the above discussion and the continuity of $\beta(x)$ and $\gamma(x)$, there exists some $C_0>0$ such that
\begin{equation}\label{302}
\left\|\gamma I_*-
\frac{\beta S_*I_*}{S_*+I_*}\right\|_{L^\infty(\Omega)}\le C_0e^{-\frac12\lambda_p(d_I)t}.
\end{equation}
Define
\begin{equation}\label{3003}
\alpha=\alpha(J,\Omega)=\inf_{u\in L^2(\Omega),\int_{\Omega}u=0,u\not\equiv0}
\frac{\frac{d_S}{2}\int_{\Omega}\int_{\Omega}J(x-y)(u(y)-u(x))^2dydx}
{\int_{\Omega}u^2(x)dx}.
\end{equation}
Following Proposition 3.4 and Lemma 3.5 in \cite{AMRT2010}, we get
\begin{equation*}
0<\alpha\le d_S\min_{x\in\bar{\Omega}}\int_{\Omega}J(x-y)dy.
\end{equation*}
Meanwhile, by the same method of the proof of Lemma 3.5 in \cite{AMRT2010}, we can obtain that
\begin{equation*}
\lambda_p(d_I)\le\min_{\bar{\Omega}}\left(
d_I\int_{\Omega}J(x-y)dy+\gamma(x)-\beta(x)\right).
\end{equation*}
Set
$$S_*(x,t)=\hat{S}_1(x,t)+\frac{1}{|\Omega|}\int_{\Omega}S_*(x,t)dx.$$
Due to $I_*(x,t)\to0$ uniformly on $\bar{\Omega}$ as $t\to+\infty$, we know $\int_{\Omega}S_*(x,t)dx\to N$ as $t\to+\infty$. Thus, we get
\begin{equation*}
\frac{1}{|\Omega|}\int_{\Omega}S_*(x,t)dx\to \frac{N}{|\Omega|}
~~\text{as}~t\to+\infty.
\end{equation*}
Note that $\int_{\Omega}\hat{S}_1(x,t)dx=0$ and $\hat{S}_1(x,t)$ satisfies
\begin{equation}\label{303}
\frac{\partial \hat{S}_1(x,t)}{\partial t}=d_S\int_{\Omega}J(x-y)
(\hat{S}_1(y,t)-\hat{S}_1(x,t))dy+f(x,t),~~x\in\Omega,~t>0,
\end{equation}
where
\begin{equation*}
f(x,t)=\gamma(x)I_*-\frac{\beta(x)S_*I_*}{S_*+I_*}
-\frac{1}{|\Omega|}\int_{\Omega}
\left(\gamma(x)I_*-\frac{\beta(x)S_*I_*}{S_*+I_*}\right)dx.
\end{equation*}
According to \eqref{302}, there exists some positive constant $c_*>0$ such that
$$|f(x,t)|\le c_*e^{-\frac12\lambda_p(d_I)t}.$$
Now, let $W(t)=\int_{\Omega}\hat{S}_1^2(x,t)dx$. Hence, the direct calculation yields that
\begin{eqnarray*}
\frac{dW(t)}{dt} &=& 2\int_{\Omega}\hat{S}_1(x,t)
\frac{\partial\hat{S}_1(x,t)}{\partial t}dx\\
&=& 2\int_{\Omega}\hat{S}_1(x,t)\left[d_S
\int_{\Omega}J(x-y)(\hat{S}_1(y,t)-\hat{S}_1(x,t))dy+f(x,t)\right]dx\\
&=& -d_S\int_{\Omega}\int_{\Omega}J(x-y)
(\hat{S}_1(y,t)-\hat{S}_1(x,t))^2dydx+
2\int_{\Omega}\hat{S}_1(x,t)f(x,t)dx\\
&\le& -2\alpha W(t)+4c_*Ne^{-\frac12\lambda_p(d_I)t}.
\end{eqnarray*}
This implies that
\begin{equation}\label{304}
\begin{aligned}
W(t)\leq& W(0)e^{-2\alpha t}+ce^{-2\alpha t}\int_{0}^{t}
e^{(2\alpha-\frac12\lambda_p(d_I))s}ds\\
=&
\begin{cases}
(W(0)+ct)e^{-2\alpha t}~& \text{if}~\lambda_p(d_I)=4\alpha,\\
c_1e^{-2\alpha t}+c_2e^{-\frac12\lambda_p(d_I) t}~& \text{if}~\lambda_p(d_I)\neq4\alpha
\end{cases}
\end{aligned}
\end{equation}
for some positive constants $c, c_1$ and $c_2$. On the other hand, it follows from \eqref{303} that
\begin{equation}\label{305}
\hat{S}_1(x,t)=\hat{S}_1(x,0)e^{-a(x)t}+e^{-a(x)t}\int_{0}^{t}
e^{a(x)s}\left[d_S\int_{\Omega}J(x-y)\hat{S}_1(y,s)dy+f(x,s)\right]ds,
\end{equation}
where $a(x)=d_S\int_{\Omega}J(x-y)dy$. By H\"{o}lder inequality, we have
\begin{equation}\label{306}
\int_{\Omega}J(x-y)\hat{S}_1(y,s)dy\le c_3W^{\frac12}(t)
\end{equation}
for some positive constant $c_3$. Then, combining \eqref{304}-\eqref{306}, it can be obtained that
\begin{equation*}
|\hat{S}_1(x,t)|\to0~~\text{as}~~t\to+\infty.
\end{equation*}
Consequently, we have
\begin{equation*}
S_*(x,t)\to\frac{N}{|\Omega|}~~\text{uniformly on}~\bar{\Omega}~\text{as}~t\to+\infty.
\end{equation*}
This completes the proof.
\end{proof}

\begin{remark}{\rm
Theorem \ref{theorem302} implies that when $R_0<1$, the epidemic disease will be extinct.
}
\end{remark}

\section{The endemic equilibrium}\label{sec4}
\noindent

In this section, we consider the existence, uniqueness of the positive solutions of \eqref{301} which is the so-called endemic equilibrium of \eqref{101}. Also, the longtime behavior of positive solutions of \eqref{101} is discussed.

\begin{lemma}\label{lemma401}
The pair of $(\tilde{S}(x),\tilde{I}(x))$ is a solution of \eqref{301} if and only if $(\tilde{S}(x),\tilde{I}(x))$ is a solution of
\begin{equation*}
\begin{aligned}
&k=d_S\tilde{S}+d_I\tilde{I},~x\in\Omega,\\
&d_I\int_{\Omega}J(x-y)(\tilde{I}(y)-\tilde{I}(x))dy
+\frac{\beta(x)\tilde{S}\tilde{I}}{\tilde{S}+\tilde{I}}
-\gamma(x)\tilde{I}=0,~x\in\Omega,\\
&N=\int_{\Omega}(\tilde{S}(x)+\tilde{I}(x))dx,
\end{aligned}
\end{equation*}
where $k$ is some positive constant.
\end{lemma}
\begin{proof}
Suppose $(\tilde{S}(x),\tilde{I}(x))$ is a solution of \eqref{301}. Then, adding the two equations of \eqref{301} yields that
\begin{equation*}
\int_{\Omega}J(x-y)[(d_S\tilde{S}(y)+d_I\tilde{I}(y))
-(d_S\tilde{S}(x)+d_I\tilde{I}(x))]dy=0,~~x\in\Omega.
\end{equation*}
Thus, there is some constant $k$ according to Proposition 3.3 in \cite{AMRT2010} such that
$$d_S\tilde{S}(x)+d_I\tilde{I}(x)=k,~x\in\Omega.$$
Meanwhile, the other cases are obvious.

In turn, if $d_S\tilde{S}(x)+d_I\tilde{I}(x)=k$, we have
\begin{equation*}
d_S\int_{\Omega}J(x-y)(\tilde{S}(y)-\tilde{S}(x))dy+
d_I\int_{\Omega}J(x-y)(\tilde{I}(y)-\tilde{I}(x))dy=0,~x\in\Omega.
\end{equation*}
Then,
\begin{equation*}
\begin{aligned}
&d_S\int_{\Omega}J(x-y)(\tilde{S}(y)-\tilde{S}(x))dy\\
=&-d_I\int_{\Omega}J(x-y)(\tilde{I}(y)-\tilde{I}(x))dy
=\frac{\beta(x)\tilde{S}\tilde{I}}{\tilde{S}+\tilde{I}}
-\gamma(x)\tilde{I},
\end{aligned}
\end{equation*}
which implies that $(\tilde{S},\tilde{I})$ satisfies \eqref{301}. This ends the proof.
\end{proof}

Let $S(x):=\frac{\tilde{S}(x)}{k}$ and $I(x):=\frac{d_I\tilde{I}(x)}{k}$, where $k$ is defined as in Lemma \ref{lemma401}. Then, we have the following result.
\begin{lemma}\label{lemma402}
The pair $(\tilde{S}(x), \tilde{I}(x))$ is a solution of \eqref{301} if and only if $(S(x),I(x))$ is a solution of
\begin{equation}\label{401}
\left\{
\begin{aligned}
&1=d_SS(x)+I(x),~x\in\Omega,\\
&0=d_I\int_{\Omega}J(x-y)(I(y)-I(x))dy+
(\beta(x)-\gamma(x))I-\frac{d_S\beta(x)I^2}{d_SI+d_I(1-I)},~x\in\Omega,\\
&k=\frac{d_IN}{\int_{\Omega}(d_IS(x)+I(x))dx}.
\end{aligned}
\right.
\end{equation}
\end{lemma}
\begin{theorem}
Suppose $R_0>1$. Then, \eqref{301} has a nonnegative solution $(\tilde{S}(x), \tilde{I}(x))$ which satisfies $\tilde{S}(x), \tilde{I}(x)\in C(\bar{\Omega})$ and $\tilde{I}(x)\not\equiv0$ on $\bar{\Omega}$. Moreover, $(\tilde{S}(x), \tilde{I}(x))$ is a unique solution of \eqref{301}, $0<\tilde{S}(x)<\frac{k}{d_S}$ and $0<\tilde{I}(x)<\frac{k}{d_I}$ for some positive constant $k$ dependent on $d_S, d_I$.
\end{theorem}
\begin{proof}
Since $R_0>1$, we can obtain that $\lambda_p(d_I)<0$ according to
Corollary \ref{remark209}. Without loss of generality, letting $m(x)=-d_I\int_{\Omega}J(x-y)dy+\beta(x)-\gamma(x)$, we can find a function sequence $\{m_n\}_{n=1}^{\infty}$ such that $\|m_n-m\|_{L^\infty(\Omega)}\to0$ as $n\to+\infty$ and the eigenvalue problem
\begin{equation*}
d_I\int_{\Omega}J(x-y)\varphi_n(y)dy+m_n(x)\varphi_n(x)
=-\lambda\varphi_n(x)~~~\text{in}~\Omega
\end{equation*}
admits a principal eigenpair $(\lambda_p^n(d_I),\varphi_n(x))$. Furthermore, taking $n$ large enough, provided $n\ge n_0$, we have
\begin{equation*}
\lambda_p^n(d_I)\le\frac12\lambda_p(d_I)-\|m_n-m\|_{L^\infty(\Omega)}.
\end{equation*}
Now, constructing $\underline{I}(x)=\delta\varphi_n(x)$ for some $\delta>0$ and the direct computation yields that
\begin{eqnarray*}
&& d_I\int_{\Omega}J(x-y)(\underline{I}(y)-\underline{I}(x))dx+
(\beta(x)-\gamma(x))\underline{I}
-\frac{d_S\beta(x)\underline{I}^2}{d_S\underline{I}+d_I(1-\underline{I})}\\
&=& -\delta\lambda_p^n(d_I)\varphi_n(x)+\delta\varphi_n(x)(m(x)-m_n(x))
-\frac{d_S\beta(x)\delta^2\varphi_n^2(x)}{d_S\delta\varphi_n(x)
+d_I(1-\delta\varphi_n(x))}\\
&\ge& -\frac12\lambda_p(d_I)\delta\varphi_n(x)
-\frac{d_S\beta(x)\delta^2\varphi_n^2(x)}
{d_S\delta\varphi_n(x)+d_I(1-\delta\varphi_n(x))}\\
&\ge& 0,
\end{eqnarray*}
provided $\delta$ small enough. Denote $\overline{I}(x)=1$. Then, it is easy to verify that
\begin{equation*}
d_I\int_{\Omega}J(x-y)(\overline{I}(y)-\overline{I}(x))dx+
(\beta(x)-\gamma(x))\overline{I}
-\frac{d_S\beta(x)\overline{I}^2}{d_S\overline{I}+d_I(1-\overline{I})}
=-\gamma(x)<0,
\end{equation*}
which implies that $\overline{I}$ is a super solution. We can take $\delta>0$ sufficiently small such that $\underline{I}\le\overline{I}$ on $\bar{\Omega}$. By the standard iteration method \cite{Coville2010,MR20091,ZLS2010}, there exists some $I(x)\in L^2(\Omega)$ satisfying \eqref{401} and $0<I(x)\le1$.

Define a functional as follows
\begin{equation*}
F(x,v)=d_I\int_{\Omega}J(x-y)I(y)dy+\left(\beta(x)-\gamma(x)
-d_I\int_{\Omega}J(x-y)dy\right)v
-\frac{d_S\beta(x)v^2}{d_Sv+d_I(1-v)}.
\end{equation*}
Obviously, $F(x,I)=0$. Note that
\begin{eqnarray*}
\frac{\partial F(x,v)}{\partial v}
&=&\beta(x)-\gamma(x)
-d_I\int_{\Omega}J(x-y)dy\\
&&-d_S\beta(x)\left\{\frac{d_Iv}
{[d_Sv+d_I(1-v)]^2}
+\frac{v}{d_Sv+d_I(1-v)}\right\}.
\end{eqnarray*}
Thus, we get
\begin{equation*}
\frac{\partial F(x,I)}{\partial v}
=-d_I\int_{\Omega}J(x-y)\frac{I(y)}{I(x)}dy
-\frac{d_Sd_I\beta(x)I(x)}{[d_SI(x)+d_I(1-I(x))]^2}<0.
\end{equation*}
And now the Implicit Function Theorem implies that $I(x)$ is continuous on $\bar{\Omega}$.

We claim that $I(x)\neq1$ for all $x\in\bar{\Omega}$. On the contrary, assume that there is $x_0\in Int(\Omega)$ such that $I(x_0)=1$. Thus, \eqref{401} yields that $\gamma(x_0)=d_I\int_{\Omega}J(x_0-y)(I(y)-I(x_0))dy\le0$, which is a contradiction. On the other hand, if $x_0\in\partial\Omega$, we can find a point sequence $\{x_n\}\subset\Omega$ such that $x_n\to x_0$ and $I(x_n)=1$, $I(x_n)\to I(x_0)$ as $n\to+\infty$. The same arguments can lead to a contradiction.

Now, we prove the uniqueness of positive solutions of \eqref{301}. For the super-sub solution pair $(I(x),1)$, we can obtain another solution of \eqref{301} by the basic iterative scheme, denoted by $I_1(x)$. Then, $I(x)\le I_1(x)\le 1$. Define
\begin{equation*}
\tau^*=\inf\{\tau>0|~I(x)\ge \tau I_1(x), ~x\in\bar{\Omega}\}.
\end{equation*}
By the boundedness of $I(x)$ and $I_1(x)$, $\tau^*$ is well defined. We claim that $\tau^*\ge1$. On the contrary, assume $\tau^*<1$. The direct calculation yields that
\begin{equation}\label{402}
\begin{aligned}
&d_I\int_{\Omega}J(x-y)(\tau^*I_1(y)-\tau^*I_1(x))dy
+(\beta(x)-\gamma(x))\tau^*I_1
-\frac{d_S\beta(x){\tau^*}^2I_1^2}{d_S\tau^*I_1+d_I(1-\tau^*I_1)}\\
=& \tau^*\beta(x)\left(\frac{d_SI_1}{d_SI_1+d_I(1-I_1)}
-\frac{d_S\tau^*I_1}
{d_S\tau^*I_1+d_I(1-\tau^*I_1)}\right)I_1>0.
\end{aligned}
\end{equation}
By the definition of $\tau^*$, there is some $x_0\in\Omega$ such that $I(x_0)=\tau^*I_1(x_0)$. Thus, we have
\begin{equation}\label{403}
\begin{aligned}
&d_I\int_{\Omega}J(x_0-y)\tau^*I_1(y)dy
-d_I\int_{\Omega}J(x_0-y)dy\tau^*
I_1(x_0)+(\beta(x_0)-\gamma(x_0))\tau^*I_1(x_0)\\
&-\beta(x_0)\frac{d_S{\tau^*}^2I_1^2(x_0)}
{d_S\tau^*I_1(x_0)+d_I(1-\tau^*I_1(x_0))}\\
=& d_I\int_{\Omega}J(x_0-y)(\tau^*I_1(y)-I(y))dy\le0.
\end{aligned}
\end{equation}
Let $\omega(y)=\tau^*I_1(y)-I(y)$ for $y\in\Omega$. Combining \eqref{402} and \eqref{403}, we have $d_I\int_{\Omega}J(x_0-y)\omega(y)dy=0$. Thus, this implies that $\omega(y)=0$ almost everywhere in $\Omega$. That is $I(x)=\tau^*I_1(x)$ almost everywhere in $\Omega$. Hence,
\begin{eqnarray*}
0&=& d_I\int_{\Omega}J(x-y)(I(y)-I(x))dy+\beta(x)
\left(1-\frac{d_SI(x)}
{d_SI(x)+d_I(1-I(x))}\right)I(x)-\gamma(x)I(x)\\
&=& \tau^*\Bigg[d_I\int_{\Omega}J(x-y)(I_1(y)-I_1(x))dy
+(\beta(x)-\gamma(x))I_1(x)
-\frac{d_S\beta(x)I_1^2(x)}
{d_SI_1(x)+d_I(1-I_1(x))}\Bigg]\\
&&+\tau^*\beta(x)\Bigg(\frac{d_SI_1(x)}
{d_SI_1(x)+d_I(1-I_1(x))}
-\frac{d_S\tau^*I_1(x)}
{d_S\tau^*I_1(x)+d_I(1-\tau^*I_1(x))}\Bigg)I_1(x)\\
&=& \tau^*\beta(x)\left(\frac{d_SI_1(x)}{d_SI_1(x)+d_I(1-I_1(x))}
-\frac{d_S\tau^*I_1(x)}
{d_S\tau^*I_1(x)+d_I(1-\tau^*I_1(x))}\right)I_1(x)>0,
\end{eqnarray*}
which is a contradiction. Thus, $\tau^*\ge1$ and this implies that $I(x)=I_1(x)$. The uniqueness of positive solutions of \eqref{301} is obtained.

Note that $S(x)=\frac{1-I(x)}{d_S}\in C(\bar{\Omega})$. Meanwhile, we get \eqref{301} admits a unique solution pair $(\tilde{S},\tilde{I})$ and $\tilde{S}(x)\in C(\bar{\Omega})$, $\tilde{I}(x)\in C(\bar{\Omega})$. Additionally, there are $0<\tilde{I}(x)<\frac{k}{d_I}$, $0<\tilde{S}(x)<\frac{k}{d_S}$ for some positive constant $k$ dependent on $d_S$ and $d_I$. The proof is complete.
\end{proof}

Below, we discuss the stability of endemic equilibrium in the sense of Definition \ref{def301}. First, we introduce a nonlocal dispersal problem as
\begin{equation}\label{404}
\begin{cases}
\frac{\partial u(x,t)}{\partial t}=d\int_{\Omega}J(x-y)(u(y,t)-u(x,t))dy+(r(x)-c(x)u)u,&~x\in\Omega,~t>0,\\
u(x,0)=u_0(x),&~x\in\Omega,
\end{cases}
\end{equation}
where $d>0$ is a positive constant and $u_0(x)$ is a bounded continuous function.
\begin{lemma}\label{lemma404}
Suppose (J) holds. Assume $r(x),c(x)\in C(\bar{\Omega})$ and $c(x)>0$ on $\bar{\Omega}$. Then the positive stationary solution $u_*$ of \eqref{404} is unique if and only if $\lambda_p(d)<0$, in which
\begin{equation*}
\lambda_p(d)=\inf_{\varphi\in L^2(\Omega),\varphi\neq0}
\frac{\frac d2\int_{\Omega}\int_{\Omega}J(x-y)(\varphi(y)-\varphi(x))^2dydx-
\int_{\Omega}r(x)\varphi^2(x)dx}{\int_{\Omega}\varphi^2(x)dx}.
\end{equation*}
Moreover, $u_*$ is globally asymptotically stable.
\end{lemma}

We can see the proof of Lemma \ref{lemma404} in \cite{SLY2014,SLW2012}. Here, we omit it.

\begin{theorem}
Suppose $d_S=d_I=d$. The following alternatives hold.
\begin{description}
\item[(i)] If $R_0<1$, then all the positive solutions of \eqref{101} converge to the disease-free equilibrium $(\frac{N}{|\Omega|}, 0)$ as $t\to+\infty$.
\item[(ii)] If $R_0>1$, then all the positive solutions of \eqref{101} converge to $(\tilde{S}(x), \tilde{I}(x))$ as $t\to+\infty$.
\end{description}
\end{theorem}
\begin{proof}
Note that (i) is contained in Theorem \ref{theorem302}. Thus, we only need to prove (ii). Let $v(x,t)=S_*(x,t)+I_*(x,t)$. Then, it follows from \eqref{101} that
\begin{equation}\label{405}
\begin{cases}
\frac{\partial v(x,t)}{\partial t}=d\int_{\Omega}J(x-y)(v(y,t)-v(x,t))dy,&~x\in\Omega, t>0,\\
\int_{\Omega}v(x,t)dx=N, &~t>0,\\
v(x,0)\ge0,&~x\in\Omega.
\end{cases}
\end{equation}
Obviously, $\frac{N}{|\Omega|}$ is the constant stationary solution of \eqref{405}. Define
\begin{equation*}
\lambda_0=\inf_{\psi\in L^2(\Omega), \int_{\Omega}\psi(x)dx=0,\psi\not\equiv0}\frac{\frac d2\int_{\Omega}\int_{\Omega}J(x-y)(\psi(y)-\psi(x))^2dydx}
{\int_{\Omega}\psi^2(x)dx}.
\end{equation*}
By the same discussion as Theorem 3.6 in \cite{AMRT2010}, we get
\begin{equation*}
\left\|v(\cdot,t)-\frac{N}{|\Omega|}\right\|_{L^\infty(\Omega)}
\le\tilde{C}e^{-\lambda_0t}
\end{equation*}
for some positive constant $\tilde{C}$. Thus, $v(x,t)\to\frac{N}{|\Omega|}$ uniformly on $\bar{\Omega}$ as $t\to+\infty$ due to $v(\cdot,t)\in C(\bar{\Omega})$. Note that $I_*(x,t)$ satisfies
\begin{equation}\label{406}
\begin{cases}
\frac{\partial I_*}{\partial t}=d\int_{\Omega}J(x-y)(I_*(y,t)-I_*(x,t))dy+(\beta(x)-\gamma(x))I_*
-\frac{\beta(x)}{v}I_*^2,& x\in\Omega, t>0,\\
\int_{\Omega}I_*(x,0)dx>0.
\end{cases}
\end{equation}
Since $v(x,t)\to\frac{N}{|\Omega|}$ uniformly on $\bar{\Omega}$ as $t\to+\infty$, for any small $\varepsilon>0$, we can find a large $T>0$ such that
\begin{equation*}
\frac{N}{|\Omega|}-\varepsilon\le v(x,t)\le\frac{N}{|\Omega|}+\varepsilon~~\text{for all}~x\in\bar{\Omega}~\text{and}~t\ge T.
\end{equation*}
Inspired by the idea in \cite{Peng2009}, we consider the following two auxiliary problems:
\begin{equation}\label{407}
\begin{cases}
\frac{\partial \overline{I}}{\partial t}=d\int_{\Omega}J(x-y)(\overline{I}(y,t)-\overline{I}(x,t))dy
+(\beta(x)-\gamma(x))\overline{I}
-\frac{\beta(x)}{\frac{N}{|\Omega|}+\varepsilon}\overline{I}^2,& x\in\Omega, t>0,\\
\overline{I}(x,T)=I_*(x,T)>0, & x\in\Omega
\end{cases}
\end{equation}
and
\begin{equation}\label{408}
\begin{cases}
\frac{\partial \underline{I}}{\partial t}=d\int_{\Omega}J(x-y)(\underline{I}(y,t)-\underline{I}(x,t))dy
+(\beta(x)-\gamma(x))\underline{I}
-\frac{\beta(x)}{\frac{N}{|\Omega|}-\varepsilon}\underline{I}^2,& x\in\Omega, t>0,\\
\underline{I}(x,T)=I_*(x,T)>0, & x\in\Omega.
\end{cases}
\end{equation}
The comparison principle implies that $\overline{I}(x,t)$ and $\underline{I}(x,t)$ are respectively the upper and lower solutions of \eqref{406}. Thus, we get
\begin{equation*}
\underline{I}(x,t)\le I_*(x,t)\le \overline{I}(x,t)~~\text{for all}~x\in\bar{\Omega}~\text{and}~t\ge T.
\end{equation*}
Since $R_0>1$, we have $\lambda_p(d_I)<0$. According to Lemma \ref{lemma404}, there are two positive functions $\overline{I}_\varepsilon(x)$ and $\underline{I}_\varepsilon(x)\in C(\bar{\Omega})$ such that
\begin{equation*}
\overline{I}(x,t)\to\overline{I}_\varepsilon(x)~\text{and}
~\underline{I}(x,t)\to\underline{I}_\varepsilon(x)~\text{uniformly on}~\bar{\Omega}~\text{as}~t\to+\infty,
\end{equation*}
and $\overline{I}_\varepsilon(x)$, $\underline{I}_\varepsilon(x)$ are respectively the unique steady states of \eqref{407} and \eqref{408}. That is, $\overline{I}_\varepsilon(x)$ and $\underline{I}_\varepsilon(x)$ satisfy
\begin{equation*}
d\int_{\Omega}J(x-y)(\overline{I}_\varepsilon(y)-\overline{I}_\varepsilon(x))dy
+(\beta(x)-\gamma(x))\overline{I}_\varepsilon
-\frac{\beta(x)}{\frac{N}{|\Omega|}+\varepsilon}
\overline{I}_\varepsilon^2=0,~x\in\Omega
\end{equation*}
and
\begin{equation}\label{409}
d\int_{\Omega}J(x-y)(\underline{I}_\varepsilon(y)
-\underline{I}_\varepsilon(x))dy
+(\beta(x)-\gamma(x))\underline{I}_\varepsilon
-\frac{\beta(x)}{\frac{N}{|\Omega|}-\varepsilon}
\underline{I}_\varepsilon^2=0,~x\in\Omega,
\end{equation}
respectively. By the same arguments in \cite{SLW2012}, we know there exists some constant $M$ independent of $\varepsilon$ such that $\underline{I}_\varepsilon(x)\le M$ and $\overline{I}_\varepsilon(x)\le M$ for all $x\in\Omega$. Additionally, $\overline{I}_\varepsilon(x)$ and $\underline{I}_\varepsilon(x)$ are monotone with respect to $\varepsilon$. In fact, assume $\varepsilon_1<\varepsilon_2$, $\underline{I}_{\varepsilon_1}(x)$ and $\underline{I}_{\varepsilon_2}(x)$ are respectively the solutions of \eqref{409} as $\varepsilon=\varepsilon_1$ and $\varepsilon=\varepsilon_2$. The direct computation yields that
\begin{equation*}
\begin{aligned}
& d\int_{\Omega}J(x-y)(\underline{I}_{\varepsilon_1}(y)
-\underline{I}_{\varepsilon_1}(x))dy
+(\beta(x)-\gamma(x))\underline{I}_{\varepsilon_1}
-\frac{\beta(x)}{\frac{N}{|\Omega|}-\varepsilon_2}
\underline{I}_{\varepsilon_1}^{2}\\
=&\frac{\beta(x)}{\frac{N}{|\Omega|}-\varepsilon_1}
\underline{I}_{\varepsilon_1}^{2}
-\frac{\beta(x)}{\frac{N}{|\Omega|}-\varepsilon_2}
\underline{I}_{\varepsilon_1}^{2}<0.
\end{aligned}
\end{equation*}
By the uniqueness of positive solution of \eqref{409}, we get $\underline{I}_{\varepsilon_2}(x)<\underline{I}_{\varepsilon_1}(x)$ for $x\in\Omega$. Meanwhile, the same arguments lead us to obtain that $\overline{I}_\varepsilon(x)$ is strictly increasing on $\varepsilon$. Now, there exists a sequence $\{\varepsilon_n\}$ with $\varepsilon_n\to0$ as $n\to+\infty$ such that
\begin{equation*}
\underline{I}_{\varepsilon_n}(x)\to I_1(x)~~\text{as}~n\to+\infty~\text{uniformly on}~\bar{\Omega}
\end{equation*}
and
\begin{equation*}
\overline{I}_{\varepsilon_n}(x)\to I_2(x)~~\text{as}~n\to+\infty~\text{uniformly on}~\bar{\Omega}
\end{equation*}
for some positive continuous functions $I_1(x)$ and $I_2(x)$. Note that $I_1(x)$ and $I_2(x)$ satisfy the following equation
\begin{equation}\label{410}
d\int_{\Omega}J(x-y)(u(y)-u(x))dy+(\beta(x)-\gamma(x))u(x)
-\frac{\beta(x)}{\frac{N}{|\Omega|}}u^2(x)=0~~\text{in}~\Omega.
\end{equation}
Then, following Lemma \ref{lemma404}, we know $I_1(x)=I_2(x)$ due to the uniqueness of positive solutions of \eqref{410}. Thus, we get that $I_*(x,t)\to I_1(x)$ uniformly on $\bar{\Omega}$ as $t\to+\infty$ and $S_*(x,t)\to \frac{N}{|\Omega|}-I_1(x)$ as $t\to+\infty$. By the uniqueness of positive solutions of \eqref{301}, we have $\tilde{S}(x)=\frac{N}{|\Omega|}-I_1(x)$ and $\tilde{I}(x)=I_1(x)$. This completes the proof.
\end{proof}

\begin{theorem}
Assume $\beta(x)=r\gamma(x)$ on $\bar{\Omega}$ for some positive constant $r\in(0,+\infty)$.
\begin{description}
\item[(i)] If $r\le1$, then the disease-free equilibrium is globally asymptotically stable;
\item[(ii)] If $r>1$, then the endemic equilibrium is globally attractive.
\end{description}
\end{theorem}
\begin{proof}
If $r<1$, then we can get $\lambda_p(d_I)>0$ by the definition of $\lambda_p(d_I)$. In this case, the result is obtained in Theorem \ref{theorem302}. Thus, we only need discuss the case $r=1$, that is $\beta(x)=\gamma(x)$. In this case, $\lambda_p(d_I)=0$, see \cite{MR2009na}. Consequently, system \eqref{101} is equivalent to
\begin{equation}\label{411}
\begin{cases}
\frac{\partial S(x,t)}{\partial t}=d_S\int_{\Omega}J(x-y)(S(y,t)-S(x,t))dy+\frac{\beta(x)I^2(x,t)}
{S(x,t)+I(x,t)} &~\text{in}~\Omega\times(0,+\infty),\\
\frac{\partial I(x,t)}{\partial t}=d_I\int_{\Omega}J(x-y)(I(y,t)-I(x,t))dy-\frac{\beta(x)I^2(x,t)}
{S(x,t)+I(x,t)} &~\text{in}~\Omega\times(0,+\infty),\\
\int_{\Omega}(S(x,t)+I(x,t))dx=N &~\text{in}~(0,+\infty),\\
S(x,0)=S_0(x)\ge0,~I(x,0)=I_0(x)\ge0 &~\text{in}~\Omega.
\end{cases}
\end{equation}
Note that the solution of \eqref{411} satisfies $\int_{\Omega}(S_*(x,t)+I_*(x,t))dx=N$ for all $t\ge0$. Thus, by the continuity of $S_*(x,t), I_*(x,t)$ with respect to $x$ and $t$ on $\bar{\Omega}\times(0,+\infty)$, we have
\begin{equation*}
\|S_*(\cdot,t)\|_{L^\infty(\Omega)}\le C_0~~\text{and}~~\|I_*(\cdot,t)\|_{L^\infty(\Omega)}\le C_0
\end{equation*}
for some positive constant $C_0$. Hence, by the same method in \cite{Peng2009}, we can obtain that
\begin{equation*}
I_*(x,t)\to0~~\text{uniformly on}~\bar{\Omega}~\text{as}~t\to+\infty.
\end{equation*}
Hence, $\int_{\Omega}S_*(x,t)dx\to N$ as $t\to+\infty$. Let
\begin{equation*}
S_*(x,t)=S_1(x,t)+\frac{1}{|\Omega|}\int_{\Omega}S_*(x,t)dx.
\end{equation*}
The direct computation yields that $S_1(x,t)$ satisfies
\begin{equation}\label{412}
\frac{\partial S_1}{\partial t}=d_S\int_{\Omega}J(x-y)(S_1(y,t)-S_1(x,t))dy
+f(x,t)
\end{equation}
for $x\in\Omega, t>0$ and $\int_{\Omega}S_1(x,t)dx=0$, in which
\begin{equation*}
f(x,t)=\frac{\beta(x)I_*^2(x,t)}
{S_*(x,t)+I_*(x,t)}-\frac{1}{|\Omega|}\int_{\Omega}\frac{\beta(x)I_*^2(x,t)}
{S_*(x,t)+I_*(x,t)}dx.
\end{equation*}
Obviously, we have $\lim\limits_{t\to+\infty}f(x,t)=0$. Note that
\begin{eqnarray*}
\int_{\Omega}S_1(x,t)f(x,t)dx &=& \int_{\Omega}\frac{\beta(x)S_1(x,t)I_*^2(x,t)}
{S_*(x,t)+I_*(x,t)}dx:=g(t).
\end{eqnarray*}
Hence, there is
\begin{equation*}
|g(t)|\le \int_{\Omega}\left|\frac{\beta(x)S_1(x,t)I_*^2(x,t)}
{S_*(x,t)+I_*(x,t)}\right|dx\le C\int_{\Omega}|S_1(x,t)|I_*(x,t)dx\le \tilde{C}\int_{\Omega}I_*(x,t)dx
\end{equation*}
for some positive constant $\tilde{C}$. Thus, we have $\lim\limits_{t\to+\infty}g(t)=0$. Let $\alpha$ be defined as \eqref{3003} and $h(x)=d_S\int_{\Omega}J(x-y)dy$. Then, following from \cite[Lemma 3.5]{AMRT2010}, there is $0<\alpha\le \min\limits_{\bar{\Omega}}h(x)$.
Define $U(t)=\int_{\Omega}S_1^2(x,t)dx$. By direct calculation, we get
\begin{eqnarray*}
\frac{dU(t)}{dt}&=& 2\int_{\Omega}S_1(x,t)\frac{\partial S_1(x,t) }{\partial t}dx\\
&=& 2\int_{\Omega}S_1(x,t)\left[d_S\int_{\Omega}J(x-y)
(S_1(y,t)-S_1(x,t))dy+f(x,t)\right]dx\\
&\le& -d_S\int_{\Omega}\int_{\Omega}J(x-y)(S_1(y,t)-S_1(x,t))^2dydx
+2g(t)\\
&\le& -2\alpha U(t)+2g(t).
\end{eqnarray*}
Thus, we have
\begin{equation*}
U(t)\le U(0)e^{-2\alpha t}+2e^{-2\alpha t}\int_{0}^{t}e^{2\alpha s}g(s)ds.
\end{equation*}
That is
\begin{equation*}
\|S_1(\cdot,t)\|_{L^2(\Omega)}\le \left(U(0)e^{-2\alpha t}+2e^{-2\alpha t}\int_{0}^{t}e^{2\alpha s}g(s)ds\right)^{\frac12}.
\end{equation*}
This implies that $\lim\limits_{t\to+\infty}\|S_1(\cdot,t)\|_{L^2(\Omega)}=0$. On the other hand, following \eqref{412}, there is
\begin{equation*}
S_1(x,t)=e^{-h(x)t}S_1(x,0)+e^{-h(x)t}\int_0^te^{h(x)s}\left[
d_S\int_{\Omega}J(x-y)S_1(y,s)dy+f(x,s)\right]ds.
\end{equation*}
Note that
\begin{equation*}
\lim_{t\to+\infty}e^{-h(x)t}\int_0^te^{h(x)s}f(x,s)ds
=\lim_{t\to+\infty}\frac{f(x,t)}{h(x)}=0
\end{equation*}
and
\begin{equation*}
e^{-h(x)t}\int_0^te^{h(x)s}
\int_{\Omega}J(x-y)|S_1(y,s)|dyds\le Ce^{-h(x)t}
\int_0^te^{h(x)s}\|S_1(\cdot,s)\|_{L^2(\Omega)}ds
\end{equation*}
for some positive constant $C$. Thus, we get $\lim\limits_{t\to+\infty}|S_1(x,t)|=0$ for all $x\in\bar{\Omega}$. This implies that $\lim\limits_{t\to+\infty}S_*(x,t)=\frac{N}{|\Omega|}$ uniformly on $\bar{\Omega}$ as $t\to+\infty$. On the other hand, noticed that system \eqref{411} is quasi-monotone increasing, then the same arguments in \cite{Peng2009} can get that $(\frac{N}{|\Omega|}, 0)$ is asymptotically stable.
This ends the proof of (i).

Below, define a functional as follows
\begin{equation*}
V(t):=V(S,I)(t)=\frac12\int_{\Omega}\left[\frac{(S-\tilde{S})^2}
{\tilde{S}}+\frac{(I-\tilde{I})^2}{\tilde{I}}\right]dx.
\end{equation*}
Obviously, $V(\tilde{S},\tilde{I})(t)=0$ and $V(S,I)(t)>0$ if $S(x,t)\neq\tilde{S}(x)$ and $I(x,t)\neq\tilde{I}(x)$.
Note $\int_{\Omega}(\tilde{S}(x)+\tilde{I}(x))dx=N$, thus by the uniqueness of stationary solutions of \eqref{301}, we can give it explicitly by
\begin{equation*}
(\tilde{S},\tilde{I})=\left(\frac1r\frac{N}{|\Omega|},
\frac{r-1}{r}\frac{N}{|\Omega|}\right).
\end{equation*}
Consequently, there is
\begin{equation*}
\frac1r=\frac{\tilde{S}}{\tilde{S}+\tilde{I}}.
\end{equation*}
Now, differentiating $V$ with respect to $t$ yields that
\begin{eqnarray*}
\frac{dV(t)}{dt}&=& \int_{\Omega}\left[\frac{S-\tilde{S}}{\tilde{S}}\frac{\partial S(x,t)}{\partial t}+\frac{I-\tilde{I}}{\tilde{I}}\frac{\partial I(x,t)}{\partial t}\right]dx\\
&=& \int_{\Omega}\frac{S-\tilde{S}}{\tilde{S}}\left[
d_S\int_{\Omega}J(x-y)(S(y,t)-S(x,t))dy\right]dx\\
&& +\int_{\Omega}\frac{I-\tilde{I}}{\tilde{I}}\left[
d_I\int_{\Omega}J(x-y)(I(y,t)-I(x,t))dy\right]dx\\
&& +\int_{\Omega}\frac{S-\tilde{S}}{\tilde{S}}
\left(-\frac{\beta SI}{S+I}+\gamma I\right)dx+\int_{\Omega}\frac{I-\tilde{I}}{\tilde{I}}
\left(\frac{\beta SI}{S+I}-\gamma I\right)dx\\
&=& -\frac{d_S}{2\tilde{S}}\int_{\Omega}\int_{\Omega}J(x-y)
(S(y,t)-S(x,t))^2dydx\\
&& -\frac{d_I}{2\tilde{I}}\int_{\Omega}\int_{\Omega}J(x-y)
(I(y,t)-I(x,t))^2dydx
-\int_{\Omega}\gamma I\frac{(\tilde{S}I-S\tilde{I})^2}{\tilde{S}^2\tilde{I}(S+I)}dx\\
&<& 0.
\end{eqnarray*}
Then, by the standard Lyapunov stability theorem and the continuity of the solutions of \eqref{101}, we obtain that
\begin{equation*}
(S(x,t),I(x,t))\to(\tilde{S}(x),\tilde{I}(x))~~\text{uniformly on}~\bar{\Omega}~\text{as}~t\to+\infty.
\end{equation*}
That is, $(\tilde{S}(x),\tilde{I}(x))$ is globally attractive. The proof of (ii) is complete.
\end{proof}

\begin{remark}{\rm
Note that when $d_S=d_I$, the epidemic disease will persist as $R_0>1$. Moreover, if the rate of the disease transmission is proportional to the rate of the disease recovery (i.e $\beta(x)=r\gamma(x)$ for some positive constant $r$), then the epidemic disease will be completely extinct finally as $R_0\le1$ (i.e $r\le1$), and will be persistence as $R_0>1$ (i.e $r>1$). Particularly, the epidemic disease will be always extinct as $R_0<1$. The case when $d_S\neq d_I$ and $R_0>1$ is very complicate. This is a challenging work and we will leave it for further study.
}
\end{remark}


\section{The effect of the large diffusion rates}
\noindent

In this section, we discuss the effect of the large diffusion rate on the transmission of the disease. Throughout this section, we always assume that $\int_{\Omega}\beta(x)dx>\int_{\Omega}\gamma(x)dx$.
Following Corollary \ref{newcor}, we know $R_0>1$ for all $d_I>0$ in this condition. Then, the positive solution $(\tilde{S},\tilde{I})$ of \eqref{301} exists.

\begin{theorem}\label{theorem501}
If we let $d_S, d_I\to+\infty$, then
\begin{equation*}
(\tilde{S},\tilde{I})\to\left(\frac{N}{|\Omega|}\frac{\int_{\Omega}\gamma(x)dx}
{\int_{\Omega}\beta(x)dx},\frac{N}{|\Omega|}
\left(1-\frac{\int_{\Omega}\gamma(x)dx}{\int_{\Omega}\beta(x)dx}
\right)\right).
\end{equation*}
\end{theorem}
\begin{proof}
Arguing as above, we know if $(\tilde{S},\tilde{I})$ is the solution of \eqref{301}, then $\tilde{S},\tilde{I}\in C(\bar{\Omega})$. Since $\int_{\Omega}(\tilde{S}(x)+\tilde{I}(x))dx=N$, the continuity of $\tilde{S}, \tilde{I}$ gives that
\begin{equation*}
\|\tilde{S}(\cdot)\|_{L^\infty(\Omega)}\le\tilde{M}~~\text{and}~~
\|\tilde{I}(\cdot)\|_{L^\infty(\Omega)}\le\tilde{M},
\end{equation*}
where $\tilde{M}$ is a positive constant independent of $d_S$ and $d_I$.

Choosing sequences $\{d_{S,n}\}_{n=1}^{\infty}$ and $\{d_{I,n}\}_{n=1}^{\infty}$ with $d_{S,n}\to+\infty$ and $d_{I,n}\to+\infty$ as $n\to+\infty$. Meanwhile, the corresponding solution of \eqref{301} is $(\tilde{S}_n,\tilde{I}_n)$. Thus, there are subsequences still denoted by $\tilde{S}_n$ and $\tilde{I}_n$, and $\tilde{S}_*, \tilde{I}_*$ such that
\begin{equation*}
\tilde{S}_n(x)\to\tilde{S}_*(x)~~\text{and}~~
\tilde{I}_n(x)\to\tilde{I}_*(x)~~\text{weakly in}~L^2(\Omega).
\end{equation*}
Note that
\begin{equation*}
\left\|\frac{\beta(\cdot)\tilde{I}_n(\cdot)\tilde{S}_n(\cdot)}
{\tilde{I}_n(\cdot)+\tilde{S}_n(\cdot)}-\gamma(\cdot)
\tilde{I}_n(\cdot)\right\|_{L^\infty(\Omega)}\le C_*
\end{equation*}
for some positive constant $C_*$ dependent only on $\beta,\gamma$ and $\Omega$. Let
\begin{equation*}
g_n(x)=\frac{\beta(x)\tilde{I}_n(x)\tilde{S}_n(x)}
{\tilde{I}_n(x)+\tilde{S}_n(x)}-\gamma(x)\tilde{I}_n(x).
\end{equation*}
Then, $\tilde{S}_n(x)$ and $\tilde{I}_n(x)$ satisfy
\begin{equation}\label{501}
\tilde{S}_n(x)=\left[\int_{\Omega}J(x-y)dy\right]^{-1}
\left[\int_{\Omega}J(x-y)\tilde{S}_n(y)dy-\frac{g_n(x)}{d_{S,n}}\right]
\end{equation}
and
\begin{equation}\label{502}
\tilde{I}_n(x)=\left[\int_{\Omega}J(x-y)dy\right]^{-1}
\left[\int_{\Omega}J(x-y)\tilde{I}_n(y)dy+\frac{g_n(x)}
{d_{I,n}}\right],
\end{equation}
respectively. It is well-known that
\begin{equation*}
\int_{\Omega}J(x-y)\tilde{S}_n(y)dy\to
\int_{\Omega}J(x-y)\tilde{S}_*(y)dy,~
\int_{\Omega}J(x-y)\tilde{I}_n(y)dy\to
\int_{\Omega}J(x-y)\tilde{I}_*(y)dy
\end{equation*}
for all $x\in\Omega$ as $n\to+\infty$. Thus, following from \eqref{501} and \eqref{502}, we have
\begin{equation*}
\tilde{S}_n(x)\to\tilde{S}_*(x)~~\text{and}~~
\tilde{I}_n(x)\to\tilde{I}_*(x)~~\text{ in}~C(\bar{\Omega})~\text{as}~n\to+\infty.
\end{equation*}
On the other hand, $\tilde{S}_n(x)$ and $\tilde{I}_n(x)$ satisfy
\begin{equation*}
\int_{\Omega}J(x-y)(\tilde{S}_n(y)-\tilde{S}_n(x))dy
=\frac{g_n(x)}{d_{S,n}}
\end{equation*}
and
\begin{equation*}
\int_{\Omega}J(x-y)(\tilde{I}_n(y)-\tilde{I}_n(x))dy
=-\frac{g_n(x)}{d_{I,n}}.
\end{equation*}
Thus, $\tilde{S}_*(x)$ and $\tilde{I}_*(x)$ satisfy
\begin{equation*}
\int_{\Omega}J(x-y)(\tilde{S}_*(y)-\tilde{S}_*(x))dy=0~\text{and}~
\int_{\Omega}J(x-y)(\tilde{I}_*(y)-\tilde{I}_*(x))dy=0
\end{equation*}
for $x\in\Omega$, respectively. This implies that $\tilde{S}_*(x)$ and $\tilde{I}_*(x)$ are all constants, still denoted by $\tilde{S}_*$ and $\tilde{I}_*$ for the convenience.

Below, we need to show that $\tilde{S}_*$ and $\tilde{I}_*$ are all positive.

 CaseI: Assume $\tilde{I}_*=0, \tilde{S}_*>0$.
Let $\hat{I}_n(x)=\frac{\tilde{I}_n(x)}
{\|\tilde{I}_n(\cdot)\|_{L^\infty(\Omega)}}$. Thus, $\hat{I}_n(x)$ satisfies
\begin{equation}\label{503}
d_{I,n}\int_{\Omega}J(x-y)(\hat{I}_n(y)-\hat{I}_n(x))dy
+\frac{\beta(x)\tilde{S}_n(x)\hat{I}_n(x)}{\tilde{S}_n(x)+\tilde{I}_n(x)}
-\gamma(x)\hat{I}_n(x)=0~~\text{in}~\Omega.
\end{equation}
The same arguments as above yield that $\hat{I}_n(x)\to1$ as $n\to+\infty$ for all $x\in\Omega$. Integrating both sides of \eqref{503} over $\Omega$ and letting $n\to+\infty$, we have $\int_{\Omega}\beta(x)dx=\int_{\Omega}\gamma(x)dx$,
which is a contradiction.

CaseII: Assume $\tilde{I}_*>0, \tilde{S}_*=0$.
Integrating \eqref{503} on $\Omega$ and letting $n\to+\infty$, we have a contradiction with $-\int_{\Omega}\gamma(x)dx=0$.

CaseIII: Assume $\tilde{I}_*=0, \tilde{S}_*=0$. This is impossible because of $\int_{\Omega}(\tilde{S}_n(x)+\tilde{I}_n(x))dx=N$.

Thus, we get $\tilde{S}_*>0$ and $\tilde{I}_*>0$. Meanwhile, we know
\begin{equation*}
\int_{\Omega}\frac{\beta(x)\tilde{S}_*\tilde{I}_*}
{\tilde{S}_*+\tilde{I}_*}dx
=\int_{\Omega}\gamma(x)\tilde{I}_*dx
~~\text{and}~~\tilde{S}_*+\tilde{I}_*=\frac{N}{|\Omega|}.
\end{equation*}
Hence, the direct computation gives that
\begin{equation*}
\tilde{S}_*=\frac{N\int_{\Omega}\gamma(x)dx}
{|\Omega|\int_{\Omega}\beta(x)dx},~\tilde{I}_*=\frac{N}{|\Omega|}
\left(1-\frac{\int_{\Omega}\gamma(x)dx}{\int_{\Omega}\beta(x)dx}\right).
\end{equation*}
This completes the proof.
\end{proof}

\begin{theorem}
If $d_S\to+\infty$, then
\begin{equation*}
(\tilde{S}(x), \tilde{I}(x))\to\left(\frac{d_IN}{\int_{\Omega}(d_I+\theta_*(x))dx},
\frac{N\theta_*(x)}{\int_{\Omega}(d_I+\theta_*(x))dx}\right),
\end{equation*}
where $\theta_*(x)$ is the unique positive solution of the following problem
\begin{equation}\label{504}
d_I\int_{\Omega}J(x-y)(u(y)-u(x))dy+(\beta(x)-\gamma(x))u
-\frac{\beta(x)u^2}{d_I+u}=0~~\text{in}~\Omega.
\end{equation}
\end{theorem}
\begin{proof}
Inspired by the method in \cite{Peng20091}, let $\theta(x)=d_SI(x)$. Then, according to \eqref{401}, we have
\begin{equation}\label{505}
d_I\int_{\Omega}J(x-y)(\theta(y)-\theta(x))dy+(\beta(x)-\gamma(x))\theta
-\frac{\beta(x)\theta^2}{\theta+d_I( 1-d_{S}^{-1}\theta)}=0~~\text{in}~\Omega.
\end{equation}
Note that the positive solution $\theta(x)$ of \eqref{505} is monotone increasing on $d_S$. Indeed, for any $d_{S_1}<d_{S_2}$, letting $\theta_1(x)$ and $\theta_2(x)$  be solutions of \eqref{505} corresponding to $d_S=d_{S_1}$ and $d_S=d_{S_2}$ respectively, then there is
\begin{equation*}
\begin{aligned}
& d_I\int_{\Omega}J(x-y)(\theta_1(y)-\theta_1(x))dy
+(\beta(x)-\gamma(x))\theta_1
-\frac{\beta(x)\theta_1^2}{\theta_1+d_I( 1-d_{S_2}^{-1}\theta_1)}\\
=&\frac{\beta(x)\theta_1^2}{\theta_1+d_I( 1-d_{S_1}^{-1}\theta_1)}
-\frac{\beta(x)\theta_1^2}{\theta_1+d_I( 1-d_{S_2}^{-1}\theta_1)}>0.\\
\end{aligned}
\end{equation*}
Thus, $\theta_1(x)<\theta_2(x)$ for all $x\in\Omega$. Since $\theta(x)\in C(\bar{\Omega})$, there exists some $x_0\in\bar{\Omega}$ such that $\theta(x_0)=\max\limits_{\bar{\Omega}}\theta(x)$. Then, it follows from \eqref{505} that
\begin{equation*}
(\beta(x_0)-\gamma(x_0))\theta(x_0)
-\frac{\beta(x_0)\theta^2(x_0)}{\theta(x_0)+d_I( 1-d_{S}^{-1}\theta(x_0))}\ge 0.
\end{equation*}
That is
\begin{eqnarray*}
\theta(x_0)&\le& \frac{d_I(\beta(x_0)-\gamma(x_0))}
{\gamma(x_0)}(1-d_{S}^{-1}\theta(x_0))\\
&\le& \frac{d_I(\beta(x_0)-\gamma(x_0))}
{\gamma(x_0)}.
\end{eqnarray*}
Thus, we have
\begin{equation*}
\theta(x)\le d_I\max_{\bar{\Omega}}
\left(\frac{\beta(x)-\gamma(x)}{\gamma(x)}\right).
\end{equation*}
Since $\theta(x)$ is monotone increasing and uniformly bounded on $d_S$, there exists some sequences $\{d_{S,n}\}_{n=1}^{\infty}$ satisfying $d_{S,n}\to+\infty$ as $n\to+\infty$ such that $\theta_n(x)=d_{S,n}I_n(x)\to \theta_*(x)$ in $C(\bar{\Omega})$ for some nonnegative function $\theta_*(x)$ as $n\to+\infty$, where $\theta_n(x)$ is the solution of \eqref{505} with $d_S=d_{S,n}$. Thus, $\theta_*(x)$ is the unique positive solution of \eqref{504}. We claim that $\theta_*(x)\neq0$. On the contrary, assume that $\theta_*(x)=0$. Let
\begin{equation*}
\hat{\theta}_n(x)=\frac{\theta_n(x)}{\|\theta_n\|_{L^\infty(\Omega)}}.
\end{equation*}
Then, $\hat{\theta}_n(x)$ satisfies
\begin{equation*}
d_I\int_{\Omega}J(x-y)(\hat{\theta}_n(y)-\hat{\theta}_n(x))dy
+(\beta(x)-\gamma(x))\hat{\theta}_n
-\frac{\beta(x)\hat{\theta}_n\theta_n}{\theta_n+d_I( 1-d_{S}^{-1}\theta_n)}=0~~\text{in}~\Omega.
\end{equation*}
Note that there is some $\hat{\theta}(x)>0$ such that $\hat{\theta}_n(x)\to\hat{\theta}(x)$ as $n\to+\infty$ and $\hat{\theta}$ satisfies
\begin{equation*}
d_I\int_{\Omega}J(x-y)(\hat{\theta}(y)-\hat{\theta}(x))dy
+(\beta(x)-\gamma(x))\hat{\theta}(x)=0~~~\text{in}~\Omega.
\end{equation*}
It follows from Lemma \ref{lemma207} that $\lambda_p(d_I)=0$. This is a contradiction according to the discussion in Sections $3$ and $4$.

On the other hand, we know $\theta_n(x)=d_{S,n}I_n(x)$. Thus, there holds $I_n(x)=\frac{\theta_n(x)}{d_{S,n}}\to0$ as $n\to+\infty$. Due to $d_{S,n}S_n(x)=1-I_n(x)$, we have $d_{S,n}S_n(x)\to1$ as $n\to+\infty$. Hence, applying \eqref{401} yields that
\begin{equation*}
\tilde{S}_n(x)=kS_n(x)=\frac{d_INS_n(x)}
{\int_{\Omega}(d_IS_n(x)+I_n(x))dx}
=\frac{d_INd_{S,n}S_n(x)}{\int_{\Omega}(d_Id_{S,n}S_n(x)+d_{S,n}I_n(x))dx}
\end{equation*}
and
\begin{equation*}
\tilde{I}_n(x)=\frac{k}{d_I}I_n(x)=\frac{NI_n(x)}
{\int_{\Omega}(d_IS_n(x)+I_n(x))dx}
=\frac{Nd_{S,n}(x)I_n(x)}{\int_{\Omega}
(d_Id_{S,n}S_n(x)+d_{S,n}I_n(x))}.
\end{equation*}
Consequently, we obtain that
\begin{equation*}
\tilde{S}_n(x)\to\frac{d_IN}{\int_{\Omega}(d_I+\theta_*(x))dx}~~
\text{as}~n\to+\infty
\end{equation*}
and
\begin{equation*}
\tilde{I}_n(x)\to\frac{N\theta^*(x)}
{\int_{\Omega}(d_I+\theta^*(x))dx}~~
\text{as}~n\to+\infty.
\end{equation*}
This ends the proof.
\end{proof}

\begin{theorem}
 If $d_I\to+\infty$, then
\begin{equation*}
(\tilde{S}(x), \tilde{I}(x))\to(S^*(x), I^*)
~~~\text{in}~C(\bar{\Omega}),
\end{equation*}
where $S^*(x)$ is a positive function and $I^*$ is a positive constant. Moreover, $(S^*(x), I^*)$ satisfies
\begin{equation}\label{final}
\begin{cases}
d_S\int_{\Omega}J(x-y)(S^*(y)-S^*(x))dy+\gamma(x)I^*-
\frac{\beta(x)S^*(x)I^*}{S^*(x)+I^*}=0,&~~x\in\Omega,\\
\int_{\Omega}(S^*(x)+I^*)dx=N.
\end{cases}
\end{equation}
\end{theorem}
\begin{proof}
Choose some sequence $\{d_{I,n}\}_{n=1}^{\infty}$ satisfying $d_{I,n}\to+\infty$ as $n\to+\infty$ and let $(S_n(x),I_n(x))$ be the solutions corresponding to system \eqref{301}.
By the same discussion as in Theorem \ref{theorem501}, we have that there is some constant $I^*$ such that $I_n(x)\to I^*$ as $n\to+\infty$. On the other hand, since $S_n(x)$ is bounded, we can find some subsequence still denoted by $\{S_n\}_{n=1}^{\infty}$, weakly converges to some nonnegative function $S^*(x)$ in $L^2(\Omega)$. Now, denote
\begin{eqnarray*}
&&a(x)=d_S\int_{\Omega}J(x-y)dy,~~h_n(x)=
d_S\int_{\Omega}J(x-y)S_n(y)dy,\\
&&G_n(x)=(a(x)-\gamma(x)+\beta(x))I_n(x)-h_n(x),
~~H_n(x)=\gamma I_n^2(x)+h_n(x)I_n(x).
\end{eqnarray*}
Thus, we have
\begin{eqnarray*}
h_n(x)\to d_S\int_{\Omega}J(x-y)S^*(y)dy~~\text{as}~n\to+\infty
\end{eqnarray*}
and
\begin{equation*}
G_n(x)\to(a(x)-\gamma(x)+\beta(x))I^*-d_S\int_{\Omega}J(x-y)S^*(y)dy
~~\text{as}~n\to+\infty.
\end{equation*}
Meanwhile,
\begin{equation*}
H_n(x)\to\gamma {I^*}^2+d_SI^*\int_{\Omega}J(x-y)S^*(y)dy
~~\text{as}~n\to+\infty.
\end{equation*}
Seen from the first equation of \eqref{301}, $S_n(x)$ satisfies
\begin{equation*}
a(x)S_n^2(x)+G_n(x)S_n(x)-H_n(x)=0.
\end{equation*}
Consequently, we have
\begin{equation*}
S_n(x)=\frac{-G_n(x)+\sqrt{G_n^2(x)+4H_n(x)}}{2a(x)}.
\end{equation*}
This implies that
\begin{equation*}
S_n(x)\to S^*(x)~~~\text{in}~C(\bar{\Omega})~~\text{as}~~n\to+\infty.
\end{equation*}
Additionally, the same arguments as in Theorem \ref{theorem501} yield that $S^*(x)>0$ and $I^*>0$. Obviously, $(S^*(x),I^*)$ satisfies \eqref{final}. The proof is complete.
\end{proof}

\section{Discussion}
\noindent

In the current paper, we firstly give the basic reproduction number $R_0$ of system \eqref{101}, which is an important threshold value to discuss the dynamic behavior of \eqref{101}. We prove that the disease persists when $R_0>1$, but when $R_0<1$, the disease dies out. Moreover, we also consider the effect of the large diffusion rates for the susceptible individuals or the infected individuals on the disease transmission and find that the nonlocal movement of the susceptible individuals or infected individuals will enhance the persistence of the disease.

In Section 2, we have proved the main result Theorem \ref{theorem206}, and established the relations between $R_0$ and $\lambda_p(d_I)$ even if $\lambda_p(d_I)$ is not always a principal eigenvalue of the operator $\mathcal{M}$ defined by \eqref{201}. Note that if $\beta(x)=\beta$ and $\gamma(x)=\gamma$ are all positive constants, then the linear problem
\begin{equation*}\label{601}
-d_I\int_{\Omega}J(x-y)(u(y)-u(x))dy+\gamma u(x)=\mu\beta u(x)~~\text{in}~\Omega
\end{equation*}
admits a principal eigenpair $(\mu_p, \varphi(x))$, where $\mu_p=\frac\gamma\beta$. Thus, it follows from Lemma \ref{theorem212} that $R_0=\frac{1}{\mu_p}=\frac\beta\gamma$ in this case. By the same discussion as Sections \ref{sec3} and \ref{sec4}, we have that the disease persists if $\beta>\gamma$ and the disease dies out if $\beta<\gamma$. But when the spatial heterogeneity is concerned, we know from Corollaries \ref{newcor1} and \ref{newcor} that the disease may persist even though there are some sites such that $\beta(x)<\gamma(x)$. That is, the spatial heterogeneity can enhance the spread of the disease.
In fact, from \cite{ABL-2008}, $x$ is a low-risk site if the local disease transmission rate $\beta(x)$ is lower than the local disease recovery rate $\gamma(x)$, and the high-risk site is defined in a similar way. Meanwhile, $\Omega$ is a low-risk domain if $\int_{\Omega}\beta(x)dx\le\int_{\Omega}\gamma(x)dx$ and a high-risk domain if $\int_{\Omega}\beta(x)dx>\int_{\Omega}\gamma(x)dx$. In the view of the biological point, Corollary \ref{newcor1} implies that the disease may spread even if the habitat of the species is low-risk as long as there is some high-risk site and the movement of the infected individuals is slow. But the quick movement of the infected individuals may suppress the spread of the disease. Following Corollary \ref{newcor}, we know that the disease will always persist if the species live in a high-risk domain and be extinct if the habitat of the species is filled with the low-risk sites. We hope these results will be useful for the disease control.

Additionally, due to the lack of regularity of stationary solutions of system \eqref{101}, we only discuss the effect of the large diffusion rates of the susceptible individuals or the infected individuals on the disease transmission. For other cases, there will be left for the future work. Also, we know that the diffusive ability of the species is different, here the diffusive ability represents the diffusive rates and the dispersal distance. Thus, it is more realistic to discuss that the susceptible individuals and the infected individuals have different dispersal strategy, that is the dispersal kernel functions are distinct from each other. This problem is of interest and it may have more complex dynamic results.

\section*{Acknowledgments}

\noindent

Fei-Ying Yang was partially supported by NSF of China (11401277) and Wan-Tong Li was
partially supported by NSF of China (11271172).



\begin{thebibliography}{99}

\bibitem{ABL-2007} L.J.S. Allen, B.M. Bolker, Y. Lou, A.L. Nevai, Asymptotic profiles of the steady states for an SIS epidemic patch model, \emph{SIAM J. Appl. Math.} \textbf{67} (2007), 1283-1309.

\bibitem{ABL-2008} L.J.S. Allen, B.M. Bolker, Y. Lou, A.L. Nevai, Asymptotic profiles of the steady states for an SIS epidemic reaction-diffusion model, \emph{Discrete Contin. Dyn. Syst.}  \textbf{21}  (2008), 1-20.

\bibitem{ALL-2009} L.J.S. Allen, Y. Lou, A.L. Nevai, Spatial patterns in a discrete-time SIS patch model, \emph{J. Math. Biol.} \textbf{58} (2009), 339-375.

\bibitem{AMRT2010} F. Andreu-Vaillo, J.M. Maz$\acute{o}$n, J.D. Rossi, J.
Toledo-Melero, Nonlocal Diffusion Problems, Mathematical Surveys and
Monographs, AMS, Providence, Rhode Island, 2010.

\bibitem{Be2015}H. Berestycki, A. Coulon, J.-M. Roquejoffre, L. Rossi, The effect of a line with nonlocal diffusion on Fisher-KPP propagation, \emph{Math. Models Methods Appl. Sci.} \textbf{25} (2015) 2519-2562.

\bibitem{Be2005} H. Berestycki, F. Hamel, L. Roques, Analysis of the periodically fragmented environment model. I. Species persistence., \emph{J. Math. Biol.}  \textbf{51}  (2005), 75-113.

\bibitem{BFRW1997} P. Bates, P.C. Fife, X. Ren, X. Wang, Traveling waves
in a convolution model for phase transition, \emph{Arch. Rational Mech. Anal.} \textbf{138} (1997), 105-136.

\bibitem{BZH2007} P. Bates, G. Zhao, Existence, uniquenss and
stability of the stationary solution to a nonlocal evolution
equation arising in population dispersal, \emph{J. Math. Anal.
Appl.} \textbf{332} (2007), 428-440.

\bibitem{bates} P. Bates, On some nonlocal evolution equations arising in materials science, in: H. Brunner, X.Q. Zhao and X. Zou (Eds.), Nonlinear dynamics and evolution equations, in: Fields Inst. Commun., vol. 48, Amer. Math. Soc., Providence, RI, 2006, 13-52.

\bibitem{CDM2008} J. Coville, J. D\'{a}vila, S. Mart\'{i}nez,
Existence and uniqueness of solutions to a nonlocal equation with
monostable nonlinearity, \emph{SIAM J. Math. Anal.} \textbf{39}
(2008), 1693-1709.


\bibitem{Coville2010} J. Coville, On a simple criterion for
the existence of a principal eigenfunction of some nonlocal
operators, \emph{J. Differential Equations} \textbf{249} (2010),
2921-2953.

\bibitem{Coville-2013} J. Coville, J. D\'{a}vila, S. Mart\'{i}nez, Pulsating fronts for nonlocal dispersion and KPP nonlinearity, \emph{ Ann. Inst. H. Poincar\'{e} Anal. Non     Lin\'{e}aire}  \textbf{30} (2013), 179-223.

\bibitem{Coville-2015} J. Coville, Nonlocal refuge model with a partial control,  \emph{Discrete Contin. Dyn. Syst.}, \textbf{35} (2015), 1421-1446.

\bibitem{CCR2006} E. Chasseigne, M. Chaves, J.D. Rossi, Asymptotic
behavior for nonlocal diffusion equation, \emph{J. Math. Pures
Appl.} \textbf{86} (2006), 271-291.

\bibitem{CCEM2007} C. Cort\'{a}zar, J. Coville, M. Elgueta, S.
Mart\'{i}nez, A nonlocal inhomogeneous dispersal process, \emph{J.
Differential Equations} \textbf{241} (2007), 332-358.

\bibitem{Diekmann1990} O. Diekmann, J.A.P. Heesterbeek, J.A.J. Metz, On the definition and the computation of the basic reproduction ratio $R_0$ in models for infectious diseases in heterogeneous populations, \emph{J. Math. Biol.} \textbf{28}  (1990), 365-382.

\bibitem{Driessche2002} P. van den Driessche, J. Watmough, Reproduction numbers and sub-threshold endemic equilibria for compartmental models of disease transmission, \emph{John A. Jacquez memorial volume. Math. Biosci.}  \textbf{180}  (2002), 29-48.

\bibitem{Fife2003} P. Fife, Some nonclassical trends in parabolic and parabolic--like evolutions, in: Trends in Nonlinear Analysis, Springer, Berlin (2003), 153-191.

\bibitem{MR2009} J. Garc\'{\i}a-Meli\'{a}n, J.D. Rossi, On the principal
eigenvalue of some nonlocal diffusion problems, \emph{J.
Differential Equations} \textbf{246} (2009), 21-38.

\bibitem{MR2009na} J. Garc\'{\i}a-Meli\'{a}n, J.D. Rossi,
Maximum and antimaximum principles for some nonlocal diffusion
operators, \emph{Nonlinear Anal.} \textbf{71} (2009),
6116-6121.

\bibitem{MR20091} J. Garc\'{\i}a-Meli\'{a}n, J.D. Rossi,
A logistic equation with refuge and nonlocal diffusion, \emph{
Commun. Pure Appl. Anal.} \textbf{8} (2009), 2037-2053.

\bibitem{HMMV2003} V. Hutson, S. Martinez, K. Mischaikow, G.T. Vickers,
The evolution of dispersal, \emph{J. Math. Biol.} \textbf{47} (2003),
483-517.

\bibitem{Huang2010} W. Huang, M. Han, K. Liu, Dynamics of an SIS reaction-diffusion epidemic model for disease transmission, \emph{Math. Biosci. Eng.} \textbf{7} (2010), 51-66.

\bibitem{HSW2005} J.M. Heffernan, R.J. Smith, L.M. Wahl, Perspectives on the basic reproductive ratio, \emph{J. R. Soc. Interface,} \textbf{2}(2005), 281-293.

\bibitem{KLSH2010} C.Y. Kao, Y. Lou, W. Shen, Random dispersal
vs non-local dispersal, \emph{Discrete Contin. Dyn. Syst.}
\textbf{26} (2010), 551-596.


\bibitem{LSW2010} W.T. Li, Y.J. Sun, Z.C. Wang,
Entire solutions in the Fisher-KPP equation with nonlocal dispersal,
\emph{Nonlinear Anal. Real Word Appl.} \textbf{11} (2010), 2302--2313.

\bibitem{LZZ2015}W.T. Li, Li Zhang, G.B. Zhang, Invasion entire solutions in a competition system with nonlocal dispersal, \emph{Discrete
Contin. Dyn. Syst.} \textbf{35} (2015), 1531-1560.


\bibitem{PLL2009} S. Pan, W.T. Li, G. Lin, Travelling wave fronts
in nonlocal reaction-diffusion systems and applications, \emph{Z.
Angew. Math. Phys.}  \textbf{60} (2009), 377--392.


\bibitem{Pazy} A. Pazy, ``Semigroups of Linear Operators and Applications to Partial Differential Equations,'' Springer-Verlag New York Berlin Heidelberg Tokyo, 1983.

\bibitem{Peng2009} R. Peng, S. Liu, Global stability of the steady states of an SIS epidemic reaction-diffusion model, \emph{Nonlinear Anal.} \textbf{71} (2009), 239-247.

\bibitem{Peng20091} R. Peng, Asymptotic profiles of the positive steady state for an SIS epidemic reaction-diffusion model, I, \emph{J. Differential Equations}  \textbf{247}  (2009), 1096-1119.

\bibitem{Peng2012} R. Peng, X.Q. Zhao, A reaction-diffusion SIS epidemic model in a time-periodic environment, \emph{Nonlinearity} \textbf{25} (2012), 1451-1471.


\bibitem{Peng2013} R. Peng, F. Yi, Asymptotic profile of the positive steady state for an SIS epidemic reaction-diffusion model: effects of epidemic risk and population movement, \emph{Phys. D} \textbf{259} (2013), 8-25.

\bibitem{RN2012} N. Rawal, W. Shen, Criteria for the existence and lower bounds of principal eigenvalues of time periodic nonlocal dispersal operators and applications, \emph{J. Dynam. Diff. Eqns.} \textbf{24} (2012), 927-954.

\bibitem{SHZH2010} W. Shen, A. Zhang, Spreading speeds for monostable
equations with nonlocal dispersal in space periodic habitats,
\emph{J. Differential Equations} \textbf{15} (2010), 747-795.


\bibitem{SZ2012} W. Shen, A. Zhang, Stationary solutions and spreading speeds of nonlocal monostable equations in space periodic habitats, \emph{Proc. Amer. Math. Soc.} \textbf{140}  (2012), 1681-1696.

\bibitem{XS2015}  W. Shen, X. Xie, On principal spectrum points/principal eigenvalues of nonlocal dispersal operators and applications, \emph{Discrete Contin. Dyn. Syst.} \textbf{35} (2015), 1665-1696.

\bibitem{SLY2011} J.W. Sun, W.T. Li, F.Y. Yang, Approximate the Fokker--Planck
equation by a class of nonlocal dispersal problems, \emph{Nonlinear
Anal.} \textbf{74} (2011), 3501-3509.

\bibitem{SLY2014} J.W. Sun, F.Y. Yang, W.T. Li, A nonlocal dispersal equation arising from a selection-migration model in genetics, \emph{J. Differential Equations} \textbf{257} (2014), 1372-1402.

\bibitem{SLW2012} J.W. Sun, W.T. Li, Z.C. Wang, A nonlocal dispersal logistic model
with spatial degeneracy, \emph{Discrete Contin. Dyn. Syst.} \textbf{35} (2015), 3217-3238.

\bibitem{SLW2010} Y.J. Sun, W.T. Li, Z.C. Wang, Entire solutions in nonlocal
dispersal equations with bistable nonlinearity,  \emph{J.
Differential Equations} \textbf{251} (2011), 551--581.

\bibitem{Thieme2009} H.R. Thieme, Spectral bound and reproduction number for infinite-dimensional population structure and time heterogeneity, \emph{SIAM J. Appl. Math.} \textbf{70} (2009), 188-211.

\bibitem{Wang2002} X. Wang, Metastability and stability of patterns
in a convolution model for phase transitions, \emph{J. Differential
Equations} \textbf{183} (2002), 434-461.

\bibitem{WZ2011} W. Wang, X.Q. Zhao, A nonlocal and time-delayed reaction-diffusion model of dengue transmission, \emph{SIAM J. Appl. Math.} \textbf{71} (2011), 147-168.

\bibitem{WZ2012} W. Wang, X.Q. Zhao, Basic reproduction numbers for reaction-diffusion epidemic models, \emph{SIAM J. Appl. Dyn. Syst.} \textbf{11} (2012), 1652-1673.

\bibitem{Wu2015} S.L. Wu, S. Ruan, Entire solutions for nonlocal dispersal
equations with spatio-temporal delay: monostable case,
\emph{J. Differential Equations} \textbf{258} (2015), 2435-2470.

\bibitem{YLS2016}F.Y. Yang, W.T. Li, J.W. Sun,  Principal eigenvalues for some nonlocal eigenvalue problems and applications, \emph{Discrete
Contin. Dyn. Syst.}, in press.

\bibitem{YLW2015} F.Y. Yang, W.T. Li, Z.C. Wang, Traveling waves in a nonlocal dispersal SIR epidemic model, \emph{Nonlinear Anal. Real World Appl.} \textbf{23} (2015), 129-147.

\bibitem{YYLW2013} F.Y. Yang, Y. Li, W.T. Li, Z.C. Wang, Traveling waves in a nonlocal dispersal Kermack-McKendrick epidemic model, \emph{Discrete
Contin. Dyn. Syst. Ser. B} \textbf{18} (2013), 1969-1993.

\bibitem{ZLS2010} G.B. Zhang, W.T. Li, Y.J. Sun, Asymptotic behavior for
nonlocal dispersal equations, \emph{Nonlinear Anal.} \textbf{72}
(2010), 4466-4474.



\end{thebibliography}
\end{document}